\documentclass[]{interact}

\usepackage[english]{babel}
\usepackage[utf8]{inputenc}
\usepackage{amsmath}
\usepackage{amsfonts}
\usepackage{amsthm}
\usepackage{amssymb}
\usepackage{etex}
\usepackage{hyperref}
\usepackage{tikz}
\usepackage{tikz-cd}
\usepackage{microtype}

\newtheorem{thm}{Theorem}[section]
\newtheorem{pro}[thm]{Proposition}
\newtheorem{cor}[thm]{Corollary}
\newtheorem{lem}[thm]{Lemma}

\theoremstyle{definition}
\newtheorem{rem}[thm]{Remark}
\newtheorem{defn}[thm]{Definition}
\newtheorem{notn}[thm]{Notation}
\newtheorem{exam}[thm]{Example}

\numberwithin{equation}{section}

\newcommand{\PP}{\mathbb{P}}
\newcommand{\Aff}{\mathbb{A}}
\newcommand{\CC}{\mathbb{C}}
\newcommand{\RR}{\mathbb{R}}

\newcommand{\ZZ}{\mathbb{Z}}
\newcommand{\OO}{\mathcal{O}}
\newcommand{\I}{\mathcal{I}}
\newcommand{\J}{\mathcal{J}}
\newcommand{\E}{\mathcal{E}}
\newcommand{\F}{\mathcal{F}}
\newcommand{\G}{\mathcal{G}}
\newcommand{\Q}{\mathcal{Q}}
\newcommand{\iso}{\cong}
\newcommand{\pr}{\mathit{pr}}
\newcommand{\id}{\mathit{id}}
\newcommand{\tensor}{\otimes}
\newcommand{\res}[2]{\left.{#1}\right|_{#2}} 
\newcommand{\st}{\;\vline\;} 
\newcommand{\onto}{\twoheadrightarrow}

\newcommand{\Coh}{\mathrm{Coh}}
\newcommand{\ch}{\mathrm{ch}}
\newcommand{\td}{\mathrm{td}}
\newcommand{\rk}{\mathrm{rk}}
\newcommand{\coho}{\mathcal{H}}
\newcommand{\stab}{\mathrm{Stab}}
\newcommand{\hilb}{\mathrm{Hilb}}
\newcommand{\sym}{\mathrm{Sym}}
\newcommand{\Ext}{\mathrm{Ext}}
\newcommand{\Pic}{\mathrm{Pic}}
\newcommand{\TTor}{\mathcal{T}\!\mathit{or}}

\newcommand{\MI}{\mathcal{M}^{\mathrm{I}}}
\newcommand{\MII}{\mathcal{M}^{\mathrm{II}}}
\newcommand{\MIII}{\mathcal{M}^{\mathrm{III}}}

\DeclareMathOperator{\Hom}{Hom}
\DeclareMathOperator{\Spec}{Spec}

\begin{document}

\title{Bridgeland stability conditions and skew lines on $\PP^3$}

\author{
\name{Sammy Alaoui Soulimani\thanks{Sammy Alaoui Soulimani Email: sammyalaoui@gmail.com}
and Martin G.\ Gulbrandsen\thanks{Martin G.\ Gulbrandsen Email: martin.gulbrandsen@uis.no}}
\affil{University of Stavanger, Department of Mathematics and Physics, 4036 Stavanger, Norway}
}

\maketitle

\begin{abstract}
Inspired by Schmidt's work on twisted cubics \cite{Sch2020}, we study
wall crossings in Bridgeland stability, starting with the Hilbert scheme
$\hilb^{2m+2}(\PP^3)$ parametrizing pairs of skew lines and plane conics
union a point. We find two walls. Each wall crossing corresponds to a
contraction of a divisor in the moduli space and the contracted space
remains smooth. Building on work by Chen--Coskun--Nollet \cite{CCN2011}
we moreover prove that the contractions are $K$-negative extremal in the
sense of Mori theory and so the moduli spaces are projective.
\end{abstract}


\section{Introduction}\label{Intro}

After the introduction of Bridgeland's manifold of stability conditions
on a triangulated category \cite{Bri2007}, several applications to the
study of the birational geometry of moduli spaces have appeared: the
moduli space is viewed as parametrizing stable objects in the derived
category of some underlying variety $X$, and the question is how the
moduli space changes as the stability condition varies. This is the
topic of wall crossing in the stability manifold. We refer to
\cite{MS2017} for an overview and in particular for examples of the
success of this viewpoint in cases where $X$ is a surface. For
threefolds and notably in the case $X=\PP^3$, important progress was
made by Schmidt \cite{Sch2020}, allowing among other things a study of
wall crossings for the Hilbert scheme of twisted cubics (see also Xia
\cite{Xia2018} for further work on this case; additional examples in the
same spirit have been investigated by Gallardo--Lozano Huerta--Schmidt
\cite{GLS2018} and Rezaee \cite{Rez2021}). The case considered in the
present text is that of pairs of skew lines in $\PP^3$ and their
deformations. This is analogous to twisted cubics in the sense that a
twisted cubic degenerates to a plane nodal curve with an embedded point
much as a pair of skew lines degenerates to a pair of lines in a plane
together with an embedded point.

More precisely we study wall crossing for the Hilbert scheme
$\hilb^{2m+2}(\PP^3)$ of subschemes $Y\subset \PP^3$ with Hilbert
polynomial $2m+2$. It has two smooth components $\mathcal C$ and
$\mathcal S$: a general point in $\mathcal C$ is a conic-union-a-point
$Y = C\cup\{P\}$ and a general point in $\mathcal S$ is a pair of skew
lines $Y = L_1\cup L_2$. Note that when a line pair is deformed until
the two lines meet, the result is a pair of intersecting lines with an
embedded point at the intersection, and this can also be viewed as a
degenerate case of a conic union a point.

For an appropriately chosen Bridgeland stability condition on the
bounded derived category of coherent sheaves $D^b(\PP^3)$, the ideal
sheaves $\I_Y$ can be viewed as the stable objects with fixed
Chern character, say $v=\ch(\I_Y)$. When deforming the stability
condition, we identify two walls, separating three chambers. Getting
slightly ahead of ourselves, the situation is illustrated in Figure
\ref{fig:lambdawalls} in Section \ref{sec:walls:brdgstability}: $\alpha$
and $\beta$ are parameters for the stability conditions considered and
we restrict ourselves to the region to the immediate left in the picture
of the hyperbola $\beta^2-\alpha^2=4$ (the role of this boundary curve
is explained in Section \ref{sec:prelim:comparison}). In this region we
have the two walls $W_1$ and $W_2$ separating three chambers, labeled by
Roman numerals as in the figure. Let $\MI$, $\MII$ and $\MIII$ be the
moduli spaces of Bridgeland stable objects with Chern character $v$ in
each chamber, considered as algebraic spaces (for existence see
Piyaratne--Toda \cite[Corollary 4.23]{PT2019}).

Our first main result contains the set-theoretical description of these moduli spaces:
\begin{thm}\label{thm:main1}
\leavevmode
\begin{enumerate}
   \item[(I)] $\mathcal{M}^{\mathrm{I}}$ is $\hilb^{2m+2}(\mathbb{P}^3)$
   with its two components $\mathcal{S}$ and $\mathcal{C}$ described
   above.
   \item[(II)] $\mathcal{M}^{\mathrm{II}}$ consists of:
   \begin{enumerate}
      \item[(i)] Ideal sheaves $\I_Y$ for $Y \in
      \hilb^{2m+2}(\mathbb{P}^3)$ not contained in a plane.
      \item[(ii)] Non-split extensions $\F_{P,V}$ in 
      \begin{equation}\label{eq:FpV}
	 0\to \I_{P/V}(-2)
	 \to \F_{P,V}
	 \to \OO_{\PP^3}(-1)
	 \to 0
      \end{equation}
      for $V \subset \PP^3$ a plane and $P\in V$. Moreover, $\F_{P,V}$
      is uniquely determined up to isomorphism by the pair $(P,V)$.
   \end{enumerate}
   \item[(III)] $\mathcal{M}^{\mathrm{III}}$ consists of:
   \begin{enumerate}
      \item[(i)] Ideal sheaves $\I_Y$ for $Y \in
      \hilb^{2m+2}(\mathbb{P}^3)$ a pair of disjoint lines or a pure
      double line.
      \item[(ii)] Non-split extensions $\G_{P,V}$ in 
      \begin{equation}\label{eq:GpV}
         0\to \OO_V(-2)
	 \to \G_{P, V}
	 \to \I_P(-1)
	 \to 0
      \end{equation}
      with $P$ and $V$ as above. Moreover, $\G_{P,V}$ is uniquely
      determined by the pair $(P,V)$.
   \end{enumerate}
\end{enumerate}	
\end{thm}

To locate walls and classify stable objects we employ the method due to
Schmidt \cite{Sch2020}, which involves ``lifting'' walls from an
intermediate notion of tilt stability. Schmidt considers as an
application the Hilbert scheme $\hilb^{3m+1}(\PP^3)$: it parametrizes
twisted cubics and plane cubics union a point. This was our starting
point and we can apply many of Schmidt's results directly, although
modified or new arguments are needed as well. The end result is closely
analogous in the two cases, with two wall crossings of the same nature.
In the twisted cubic situation, however, Schmidt also finds an
additional ``final wall crossing'' where all objects are destabilized.
This has no direct analogy in our case.

Next we describe the moduli spaces geometrically, guided by the set
theoretic classification of objects above; this leads to contractions of
the two smooth components $\mathcal C$ and $\mathcal S$ of $\MI =
\hilb^{2m+2}(\PP^3)$. First introduce notation for the loci that are
destabilized by the two wall crossings according to the above
classification:

\begin{notn}\label{notn:divisors}
\leavevmode
\begin{itemize}
   \item[(a)]
   Let $E\subset \mathcal C$ be the divisor consisting of all planar $Y
   \in \mathcal C$.
   \item[(b)]
   Let $F\subset \mathcal S$ be the divisor consisting of all $Y\in
   \mathcal S$ having an embedded point.
\end{itemize}
\end{notn}

Thus the locus (II)(i) is $(\mathcal C\setminus E) \cup \mathcal S$ and
the locus (III)(i) is $\mathcal S\setminus F$. On the other hand both
loci (II)(ii) and (III)(ii) are parametrized by the incidence variety
$I\subset\PP^3\times\check\PP^3$ consisting of pairs $(P,V)$ of points
$P\in\PP^3$ inside a plane $V\in\check\PP^3$. The process of replacing
$E$ and $F$ by $I$ can be realized as contractions of algebraic spaces:
$E$ and $F$ may be viewed as projective bundles over $I$, and in Section
\ref{sec:modulispace} we apply Artin's contractibility criterion to
obtain smooth algebraic spaces $\mathcal C'$ and $\mathcal S'$ each
containing the incidence variety $I$ as a closed subspace, and
birational morphisms
\begin{align*}
   \phi&\colon \mathcal C \to \mathcal C'\\
   \psi&\colon \mathcal S \to \mathcal S'
\end{align*}
which are isomorphisms outside of $E$, respectively $F$, and restrict to
the natural maps $E\to I$, respectively $F\to I$. Moreover $E\subset
\mathcal C$ is disjoint from $\mathcal S$, so the union $\mathcal C'\cup
\mathcal S$ makes sense as the gluing together of $(\mathcal C\setminus
E)\cup \mathcal S$ and $\mathcal C'$. We can then state our second main
result:

\begin{thm}\label{thm:main2}
\leavevmode
\begin{itemize}
   \item[(a)] $\MII$ is isomorphic to $\mathcal C'\cup \mathcal S$.
   \item[(b)] $\MIII$ is isomorphic to $\mathcal S'$.
\end{itemize}
\end{thm}

To prove the theorem it suffices to treat the contracted spaces as
algebraic spaces. However, they turn out to be projective varieties: the
contractions are in fact $K$-negative extremal contractions in the sense
of Mori theory. The case of $\mathcal S \to \mathcal S'$ can be found in
previous work by Chen--Coskun--Nollet \cite{CCN2011} and in fact it
turns out that $\mathcal S'$ is a Grassmannian; see Section
\ref{sec:familychamberII}. Inspired by this work, we exhibit in Section
\ref{sec:mori} the map $\mathcal C\to \mathcal C'$ as a $K$-negative
extremal contraction. This may be contrasted with Schmidt's approach in
the twisted cubic situation \cite{Sch2020}, where projectivity of the
moduli spaces is proved by viewing them as moduli of quiver
representations.

In Section \ref{sec:prelim}, we list the background results that we
need, in particular, we briefly recall the construction of stability
conditions on threefolds, along with the notion of tilt-stability. In
Section \ref{sec:walls} we apply Schmidt's machinery to prove Theorem
\ref{thm:main1}. In Section \ref{sec:modulispace} we study universal
families and prove Theorem \ref{thm:main2}. Finally, in Section
\ref{sec:mori} we work out the Mori cone of $\mathcal C$.

We work over $\CC$. Throughout and in particular in Section
\ref{sec:modulispace}, intersections and unions of subschemes are
defined by the sum and intersection of ideals, respectively, and
inclusions and equalities between subschemes are meant in the scheme
theoretic sense. The \emph{relative ideal} of an inclusion $Z\subset Y$
of two closed subschemes of some ambient scheme is the ideal $\I_{Z/Y}
\subset \OO_Y$ defining $Z$ as a subscheme of $Y$.

\section*{Acknowledgments}

This work is supported by the Research Council of Norway under grant
no.\ 230986, and is part of the Ph.D. thesis \cite{AS2023} defended by
the first author at the university of Stavanger on the 29th of March
2023. The second author was also supported by the Swedish Research
Council under grant no.\ 2016-06596 while the author was in residence at
Institut Mittag--Leffler in Djursholm, Sweden during the fall of 2021.


\section{Preliminaries}\label{sec:prelim}

After detailing the two components of $\hilb^{2m+2}(\PP^3)$, we collect
notions and results from the literature surrounding Bridgeland stability
and wall crossings for smooth projective threefolds. There are no
original results in this section.

\subsection{The Hilbert scheme and its two components}\label{sec:twocomponents}

It is known that $\hilb^{2m+2}(\PP^3)$ has two smooth components
$\mathcal C$ and $\mathcal S$, whose general points are conics union a
point and pairs of skew lines, respectively. A quick parameter count
yields $\dim \mathcal C = 11$ and $\dim \mathcal S = 8$. We refer to Lee
\cite{Lee2000} for an overview, to Chen--Nollet \cite{CN2012} for the
smoothness of $\mathcal C$ and to Chen--Coskun--Nollet \cite{CCN2011}
for the smoothness of $\mathcal S$. In fact, the referenced works show
that $\mathcal C$ is the blowup
\begin{equation*}
   \mathcal{C} \to \PP^3\times\hilb^{2m+1}(\PP^3)
\end{equation*}
along the universal conic $\mathcal Z \subset
\PP^3\times\hilb^{2m+1}(\PP^3)$ and $\mathcal S$ is the blowup
\begin{equation*}
   \mathcal{S} \to \sym^2(G(2,4))
\end{equation*}
along the diagonal in the symmetric square of the Grassmannian $G(2,4)$
of lines in $\PP^3$. In other words, it is the Hilbert scheme
$\hilb^2(G(2,4))$ of finite subschemes in $G(2,4)$ of length two.

We pause to make some preparations regarding embedded points: by a curve
$C\subset\PP^3$ \emph{with an embedded point} at $P\in C$ we mean a
subscheme $Y\subset\PP^3$ such that $C\subset Y$ and the relative ideal
$\I_{C/Y}$ is isomorphic to $k(P)$. This makes sense even when we allow
$C$ to be singular or nonreduced. Embedded points are in bijection with
normal directions, i.e.\ lines in the normal space to $C\subset \PP^3$
at $P$: explicitly in our situation, suppose that $P=(0,0,0)$ in local
affine coordinates $x,y,z$ and $C$ is a conic defined by the ideal $I_C
= (q(x,y), z)$, where $q$ is a quadric vanishing at the origin. A normal
vector to $C$ may be viewed as a $k$-linear map
\begin{equation*}
   \phi\colon I_C/\mathfrak{m}_P I_C \to k
\end{equation*}
given by say $a = \phi(q)$ and $b = \phi(z)$. Thus $(a:b) \in \PP^1$
parametrizes normal directions and the corresponding scheme $Y$ with an
embedded point at $P$ is defined by the kernel of the induced map
$I_C\to k$, which is
\begin{equation*}
   I_Y = \mathfrak{m}_P I_C + (bq - az).
\end{equation*}

\begin{exam}\label{ex:embedded}
With notation as above, consider the degenerate conic defined by $q(x,y)
= xy$. When $(a:b) = (1:0)$ we obtain the planar scheme $Y$ in the
$xy$-plane given by
\begin{equation*}
   I_Y = (x^2y, xy^2, z) = (xy,z) \cap (z, x^2,y^2)
\end{equation*}
where the final form exhibits $Y$ as the scheme theoretic union of a
pair of lines and a thickening of the origin in the $xy$-plane. At the
other extreme $(a:b) = (0:1)$ we obtain
\begin{equation*}
   I_Y = (zy,zx,yz,z^2) = (xy,z) \cap (x,y,z)^2
\end{equation*}
which we label \emph{spatial}: the subscheme $Y\subset\PP^3$ is not
contained in any nonsingular surface since it contains a full first
order infinitesimal neighbourhood around $P$. It is easy to check that
for the remaining values of $(a:b)$ the corresponding scheme $Y$ is
neither planar nor spatial. Analogous observations hold for the double
line $q(x,y) = y^2$.
\end{exam}

With these preparations, we next list all elements $Y\in
\hilb^{2m+2}(\PP^3)$, including degenerate cases. We follow Lee
\cite[Section 3.5]{Lee2000}, to which we refer for further details and
proof that the following list is exhaustive.

The component $\mathcal C$ parametrizes subschemes $Y$ of the following
form: let $C$ be a conic in a plane $V\subset\PP^3$, possibly a union of
two lines or a planar double line. Then $Y$ is either the disjoint union
of $C$ and a point $P\in\PP^3$, or $C$ with an embedded point at $P\in
C$. If $C$ is nonsingular at $P$, embedded points correspond to normal
directions, parametrized by a $\PP^1$. Since even degenerate conics are
complete intersections, also embedded point structures at a singular or
nonreduced point $P$ form a $\PP^1$ as in Example \ref{ex:embedded}, and
among these exactly one is planar ($Y$ contained in a plane) and exactly
one is spatial ($Y$ contains the first order infinitesimal neighbourhood
of $P$ in $\PP^3$).

The component $\mathcal S$ parametrizes pairs $Y = L_1 \cup L_2$ of skew
lines, together with its degenerations. These are of the following three
types: (1) a pair of incident lines $L_1\cup L_2$ with a spatial
embedded point at the intersection point, (2) a planar double line with
a spatial embedded point, or (3) a double line in a quadric surface,
i.e. there is a line $L$ in a nonsingular quadric surface $Q$ such that
$Y$ is the effective divisor $2L\subset Q$, but viewed as a subscheme of
$\PP^3$. We label this case the \emph{pure double line}, where purity
refers to the lack of embedded components.

Clearly, then, $\mathcal C\cap \mathcal S$ consists of the incident or
planar double lines with a spatial embedded point.

\subsection{Stability conditions and walls}

Let $X$ be a smooth projective threefold over $\CC$ and fix a finite
rank lattice $\Lambda$ equipped with a homomorphism $K(X) \to \Lambda$
from the Grothendieck group of coherent sheaves modulo short exact
sequences. On $\PP^3$ we will take $\Lambda = \ZZ \oplus \ZZ \oplus
\tfrac{1}{2}\ZZ \oplus \tfrac{1}{6} \ZZ$ equipped with the Chern
character map $\ch\colon K(\PP^3)\to \Lambda$.

Recall \cite{Bri2007, BMT2014, BMS2016} that a Bridgeland stability
condition $\sigma = (\mathcal A, Z)$ on $X$ (with respect to $\Lambda$)
consists of 
\begin{itemize}
   \item[(i)] an abelian subcategory $\mathcal A \subset D^b(X)$, which
   is the heart of a bounded $t$-structure, and
   \item[(ii)] a stability function $Z$, which is a group homomorphism
   \begin{equation*}
      Z\colon \Lambda \to \CC
   \end{equation*}
   whose value on any nonzero object $\E\in\mathcal A$ is in the upper
   half plane
   \begin{equation*}
      \mathbb{H}  =  \big\{z \in \CC \st \Im z \geq 0 \big\}
      \setminus \RR_{\geq 0}.
   \end{equation*}
\end{itemize}
This is subject to a list of axioms which we will not give (see
\cite[Section 8]{BMS2016}). We may then partially order the nonzero
objects in $\mathcal A$ by their \emph{slope}
\begin{equation*}
   \lambda_\sigma = -\Re(Z)/\Im(Z) \in \RR \cup \{+\infty\}.
\end{equation*}
This yields a notion of $\sigma$-stability and $\sigma$-semistability
for objects in $\mathcal A$ in the usual way by comparing the slope of
an object with that of its sub- or quotient objects. These notions
extend to $D^b(X)$ by shifting in the sense of the $t$-structure.

Bridgeland's result \cite[Theorem 1.2]{Bri2007} gives the set
$\stab_\Lambda(X)$ of stability conditions the structure of a complex
manifold, and for a given $u\in\Lambda$, it admits a wall and chamber
structure: if an object $\E$ goes from being stable to being unstable as
the stability condition $\sigma$ varies, $\sigma$ has to pass through a
stability condition for which $\E$ is strictly semistable. This
observation leads to the definition of a wall:

\begin{defn}\label{numwallsactuwalls}
Fix $u \in \Lambda$.
\begin{enumerate}
   \item[i.] \textit{Numerical walls}:
   A numerical wall is a nontrivial proper solution set
   \begin{equation*}
      W_{v,w}=\{\sigma \in \stab_\Lambda(X) \st
      \lambda_\sigma(v)=\lambda_\sigma(w)\}
   \end{equation*}
   where $u = v+w$ in $\Lambda$.
   \item[ii.]\textit{Actual walls}:
   Let $W_{v,w}$ be a numerical wall, defined by classes $v$, $w$
   satisfying $u=v+w$. A subset $V \subset W_{v,w}$ is an actual wall if
   for each point $\sigma \in V$, there is a short exact sequence
   \begin{equation*}\label{destabseq}
      0 \to \F \to \E \to \G \to 0
   \end{equation*}
   in $\mathcal{A}$ of $\sigma$-semistable objects $\F, \E, \G$ with
   classes $u,v,w$ in $\Lambda$, respectively, such that
   $\lambda_\sigma(\F)=\lambda_\sigma(\E)=\lambda_\sigma(\G)$.
   
   When the context is clear, we drop the word ``actual" and just say
   ``wall".
\end{enumerate} 
\end{defn}

Given a union of walls, we refer to each connected component of its
complement in $\stab_\Lambda(X)$ as a \emph{chamber}. By the arguments
in \cite[Section 9]{Bri2008} there is a locally finite collection of
(actual) walls in $\stab_\Lambda(X)$, each being a closed codimension
one manifold with boundary, such that the set of stable objects in
$\mathcal A$ with class $u\in \Lambda$ remains constant within each
chamber, and there are no strictly semistable objects in a chamber. 

We say that a short exact sequence as in (ii) above \emph{defines} the
wall. Relaxing this, an unordered pair $\{\F, \G\}$ \emph{defines} the
wall if there is a short exact sequence in either direction (i.e.\ we
allow the roles of sub and quotient objects to be swapped) as in (ii).
Semistability of $\F$ and $\G$ is automatic, i.e.\ it follows from
semistability of $\E$ and the equality between slopes, but see Remark
\ref{rem:weakwalls}.

\begin{rem}
A \emph{very weak} stability condition $(\mathcal A, Z)$ is a weakening
of the above concept (see Piyaratne--Toda \cite{PT2019}) where $Z$ is
allowed to map nonzero objects in $\mathcal A$ to zero. One may define
an associated slope function $\lambda$ as before, with the convention
that $\lambda(\E) = +\infty$ also when $Z(\E)=0$. An object $\E\in
\mathcal A$ is declared to be stable if every nontrivial subobject
$\F\subsetneq \E$ satisfies $\lambda(\F) < \lambda(\E/\F)$, and
semistable when nonstrict inequality is allowed. With this definition
one avoids the need to treat cases where $Z(\F)=0$ or $Z(\E/\F)=0$
separately. We will not need to go into detail.
\end{rem}

\subsection{Construction of stability conditions on threefolds}\label{sec:prelim:constr}

We next recall the ``double tilt'' construction of stability conditions
by Bayer--Macr\`i--Toda \cite{BMT2014}. For this it is necessary to
assume that the threefold $X$ satisfies a certain ``Bogomolov
inequality'' type condition \cite[Conjecture 4.1]{BMS2016}), which is
known in several cases including $\PP^3$ \cite{Mac2014}. Fix a
polarization $H$ on $X$; on $\PP^3$ this will be a (hyper)plane.

\subsubsection{Slope stability}

Let $\beta \in \RR$. The twisted Chern character of a sheaf or a complex
$\E$ on $X$ is defined by $\ch^\beta(\E) = e^{-\beta H}\ch(\E)$. Its
homogeneous components are
\begin{align*}
   \ch_0^\beta(\E)&=\ch_0(\E)=\rk(\E)\\
   \ch_1^\beta(\E)&=\ch_1(\E)-\beta H\ch_0(\E)\\
   \ch_2^\beta(\E)&=\ch_2(\E)-\beta H\ch_1(\E)+\frac{\beta^2 H^2}{2}\ch_0(\E)\\
   \ch_3^\beta(\E)&=\ch_3(\E)-\beta H\ch_2(\E)+\frac{\beta^2 H^2}{2}\ch_1(\E)-\frac{\beta^3 H^3}{6}\ch_0(\E).
\end{align*}

The twisted slope stability function on the abelian category $\Coh(X)$
of coherent sheaves is given by
\begin{equation}
   \mu_{\beta}(\E)=
   \begin{cases}
      \frac{H^{2}ch_1^{\beta}(\E)}{H^3 ch_0^{\beta}(\E)} &
      \text{if $\rk(\E)\neq 0$,}\\
      +\infty & \text{else.}
   \end{cases}
\end{equation}
This is the slope of a very weak stability condition. Notice that
$\mu_{\beta}(\E)=\mu(\E)-\beta$, where $\mu(\E)$ is the classical slope
stability function. A sheaf $\E \in \Coh(X)$ which is (semi)stable with
respect to this very weak stability condition is called
$\mu_\beta$-(semi)stable (or slope (semi)stable).

\subsubsection{Tilt stability}
Next, define the following full subcategories of $\Coh(X)$
\begin{align*}
   \mathcal{T}_{\beta}&=\{\E \in \Coh(X) \st
   \text{Any quotient $\E \twoheadrightarrow \G$ satisfies $\mu_{\beta}(\G) > 0$} \},\\	
   \mathcal{T}^\perp_{\beta}&=\{\E \in \Coh(X) \st
   \text{Any subsheaf $\F \subset \E$ satisfies $\mu_{\beta}(\F) \leq 0$} \}.
\end{align*}
The pair $(\mathcal{T}_{\beta},\mathcal{T}^\perp_{\beta}) $ is a torsion
pair (\cite[Definition 3.2]{Bri2008}) in $\Coh(X)$. Tilt the category
$\Coh(X)$ with respect to this torsion pair and denote the obtained
heart by
$\Coh^{\beta}(X)=\left<\mathcal{T}^\perp_\beta[1],\mathcal{T}_\beta\right>$.
Thus every object $\E \in \Coh^{\beta}(X)$ fits in a short exact
sequence
\begin{equation}\label{cohoseq}
   0 \to \coho^{-1}(\E)[1]  \to \E \to \coho^{0}(\E) \to 0
\end{equation}
with $\coho^{-1}(\E) \in \mathcal{T}^\perp_\beta$ and $\coho^{0}(\E) \in
\mathcal{T}_\beta$.

Let $(\alpha,\beta) \in \RR_{>0}\times\RR$, and let
\begin{equation}
   Z^{\mathrm{tilt}}_{\alpha,\beta}(\E)
   =-Hch_2^\beta(\E)+\frac{\alpha^2}{2}H^3 ch_0^\beta(\E)+iH^2ch_1^\beta(\E).
\end{equation}
The associated slope function is
\begin{equation*}
   \nu_{\alpha,\beta}(\E)=
   \begin{cases}
   \frac{Hch_2^\beta(\E)-\frac{\alpha^2}{2}H^3 ch_0^\beta(\E)}{ H^2 ch_1^\beta(\E)}
   & \text{if $H^2 ch_1^\beta(\E)\neq 0$,}\\
   +\infty & \text{else},
   \end{cases}
\end{equation*}
(see \cite[Section 9.1]{MS2017} for more details).

By \cite[Proposition B.2 (case $B=\beta H$)]{BMS2016}, the pair
$(\Coh^{\beta}(X),Z^{\mathrm{tilt}}_{\alpha,\beta})$ is a very weak
stability condition continuously parametrized by $(\alpha,\beta)\in
\RR_{>0}\times\RR$. An object $\E \in \Coh^\beta(X)$ which is
(semi)stable with respect to this very weak stability condition is
called $\nu_{\alpha,\beta}$-(semi)stable (or tilt (semi)stable).
Moreover the parameter space $\RR_{>0}\times \RR$ admits a wall and
chamber structure \cite[Proposition B.5]{BMS2016}, in which walls are
nested semicircles centered on the $\beta$-axis, or vertical lines (we
view $\alpha$ as the vertical axis) \cite [Theorem 3.3]{Sch2020}, and a
numerical wall is either an actual wall everywhere or nowhere. We refer
to them as ``tilt-stability walls" or ``$\nu$-walls" interchangeably.

\begin{rem}\label{rem:weakwalls}
Walls in the parameter space for tilt stability are defined analogously
to Definition \ref{numwallsactuwalls}. With tilt-stability in mind we
make the following observation: let $\E$ be strictly semistable with
respect to a very weak stability condition $(\mathcal A, Z)$. By
definition this means that $\E$ is semistable and there exists a short
exact sequence
\begin{equation*}
   0 \to \F \to \E \to \G \to 0
\end{equation*}
in $\mathcal A$ such that all three objects share the same slope. This
implies that $\F$ is semistable, since any destabilizing subobject of
$\F$ would also destabilize $\E$. When very weak stability conditions
are allowed, however, $\G$ may not be semistable: this happens exactly
when $\G$ has finite slope and there is a nontrivial subobject
$\G'\subset \G$ such that $Z(\G') = 0$. In this case let $\G'\subset \G$
be the maximal such subobject: then $\G/\G'$ is semistable (it is in
fact the final factor in the Harder--Narasimhan filtration of $\G$) and
has the same slope as $\G$. Moreover the kernel $\F'$ of the composite
$\E \to \G \to \G/\G'$ has the same slope as $\F$. Thus in the short
exact sequence
\begin{equation*}
   0 \to \F' \to \E \to \G/\G' \to 0
\end{equation*}
all objects are semistable and of the same slope. So when looking for
walls we may as in Definition \ref{numwallsactuwalls} assume all objects
in the defining short exact sequence to be semistable, even in the very
weak situation.
\end{rem}

The following is the Bogomolov inequality for tilt-stability:
\begin{pro}\cite[Corollary 7.3.2]{BMT2014}\label{BGtilt}
Any $\nu_{\alpha,\beta}$-semistable object $\E \in \Coh^\beta(X)$ satisfies
\begin{equation*}
   \overline{\Delta}_H(\E)=
   \big(H^2\ch^\beta_1(\E)\big)^2-2H^3\ch^\beta_0(\E)H\ch_2^\beta(\E) \geq 0.
\end{equation*}
\end{pro}

\subsubsection{Bridgeland stability}

Define the following full subcategories of $\Coh^{\beta}(X)$:
\begin{align*}
   \mathcal{T}'_{\alpha,\beta}&=\{\E \in \Coh^\beta(X) \st \text{Any
   quotient $\E \twoheadrightarrow \G$ satisfies $\nu_{\alpha,\beta}(\G) >
   0$} \},\\	
   \mathcal{T}^{'\perp}_{\alpha,\beta}&=\{\E \in \Coh^\beta(X) \st
   \text{Any subsheaf $\F \subset \E$ satisfies $\nu_{\alpha,\beta}(\F)
   \leq 0$} \}.
\end{align*}
They form a torsion pair. Tilting $\Coh^{\beta}(X)$ with respect to this
pair yields stability conditions
$(\mathcal{A}_{\alpha,\beta}(X),Z_{\alpha, \beta,s})$ (\cite[Theorem
8.6, Lemma 8.8]{BMS2016}) on $X$, where
$\mathcal{A}_{\alpha,\beta}(X)=\left<\mathcal{T}^{'\perp}_{\alpha,\beta}[1],\mathcal{T}'_{\alpha,\beta}\right>$
and
\begin{equation}\label{eq:Z^s}
   Z_{\alpha, \beta,s}=
   -\ch_3^\beta+\alpha^2(\frac{1}{6} + s)H^2\ch_1^\beta
   +i(H\ch_2^\beta-\frac{\alpha^2}{2}H^3\ch_0^\beta).
\end{equation}
The slope function of $Z_{\alpha,\beta,s}$ is given by
\begin{equation*}
   \lambda_{\alpha,\beta,s}(\E)=
   \frac{\ch_3^\beta(\E)-\alpha^2(\frac{1}{6} + s)H^2\ch_1^\beta(\E)}
   {H\ch_2^\beta(\E)-\frac{\alpha^2}{2}H^3\ch_0^\beta(\E)}
\end{equation*}
with $\lambda_{\alpha,\beta,s}(\E)=+\infty$ when
$H\ch_2^\beta(\E)=\frac{\alpha^2}{2}H^3\ch_0^\beta(\E)$.

An object $\E \in \mathcal{A}_{\alpha, \beta}(X)$ which is (semi)stable
with respect to this stability condition is called
$\lambda_{\alpha,\beta,s}$-(semi)stable.
 
The stability conditions $(\mathcal{A}_{\alpha,\beta}(X),Z_{\alpha,
\beta,s})$ are continuously parametrized by $(\alpha,\beta,s) \in
\RR_{>0} \times \RR \times \RR_{>0}$ (\cite[Proposition 8.10]{BMS2016}).
We refer to walls in $\RR_{>0}\times\RR\times\RR_{>0}$ as
``$\lambda$-walls".

The following lemma allows us to identify moduli spaces of slope-stable
sheaves with moduli spaces of tilt-stable sheaves, given the right
conditions:
 
\begin{lem}[{\cite[Lemma 1.4]{GLS2018}}]\label{largevoltilt}
On $\PP^3$, let $v=(v_0,v_1,v_2,v_3) \in \Lambda$ satisfy $\mu_\beta(v)
> 0$ and assume $(v_0,v_1)$ is primitive.
Then an object $\E\in \Coh^\beta(\PP^3)$ with $\ch(\E)=v$ is
$\nu_{\alpha,\beta}$-stable for all $\alpha \gg 0$ if and only if $\E$
is a slope stable sheaf. 
\end{lem}

\subsection{Comparison between $\nu$-stability and $\lambda$-stability --- after Schmidt}\label{sec:prelim:comparison}

Let $\E$ be an object in $D^b(X)$. Throughout this section, let
$(\alpha_0,\beta_0)\in \RR_{>0}\times \RR$ satisfy
$\nu_{\alpha_0,\beta_0}(\E)=0$, and fix $s>0$. We shall summarize a
series of results by Schmidt \cite{Sch2020} enabling us to compare walls
and chambers with respect to $\nu$-stability with those of
$\lambda$-stability. (Looking ahead to our application illustrated in
Figures \ref{fig:nuwalls} and \ref{fig:lambdawalls}, the dashed
hyperbola is the solution set to $\nu_{\alpha,\beta}(\E) = 0$.)

Consider the following conditions on $\E$:
\begin{enumerate}
   \item $\E$ is a $\nu_{\alpha_0,\beta_0}$-stable object in
   $\Coh^{\beta_0}(X)$.
   \item $\E$ is a $\lambda_{\alpha,\beta,s}$-stable object in $\mathcal
   A^{\alpha,\beta}(X)$, for all $(\alpha,\beta)$ in an open
   neighborhood of $(\alpha_0,\beta_0)$ with
   $\nu_{\alpha,\beta}(\E) > 0$.
   \item $\E$ is a $\lambda_{\alpha,\beta,s}$-semistable object in
   $\mathcal A^{\alpha,\beta}(X)$, for all $(\alpha,\beta)$ in an open
   neighborhood of $(\alpha_0,\beta_0)$ with $\nu_{\alpha,\beta}(\E) >
   0$. \item $\E$ is a $\nu_{\alpha_0,\beta_0}$-semistable object in
   $\Coh^{\beta_0}(X)$.
\end{enumerate}
Obviously there are implications (1) $\implies$ (4) and (2) $\implies$
(3). The following says that, under a mild condition on $\ch(\E)$, there
are in fact implications
\begin{equation*}
   (1) \implies (2) \implies (3) \implies (4)
\end{equation*}
so that $\lambda$-stability in a certain sense refines $\nu$-stability:

\begin{thm}[Schmidt]\label{thm:refine}
The implication $(1) \implies (2)$ above always holds.
If $H^2\ch^{\beta_0}_1(\E) > 0$ and $\overline{\Delta}_H(\E)>0$ then
also the implication $(3) \implies (4)$ holds.
\end{thm}

For the proof we refer to Schmidt \cite{Sch2020}: the first implication
follows from Lemma 6.2 in \emph{loc.\ cit.} and the second follows from
Lemmas 6.3 and 6.4.

Schmidt furthermore compares walls for $\nu$-stability and
$\lambda$-stability, for objects $\E$ in some fixed class $v\in
\Lambda$. Let
\begin{equation}\label{triangle}
   \F \to \E \to \G \to \F[1]
\end{equation}
be a triangle in $D^b(X)$ with $\E$ in class $v$.

\begin{itemize}
   \item
   Say that \eqref{triangle} \emph{defines a $\nu$-wall} through
   $(\alpha_0,\beta_0)$ if $\F$,$\E$,$\G$ are
   $\nu_{\alpha_0,\beta_0}$-semistable objects in $\Coh^{\beta_0}(X)$ and
   $\nu_{\alpha_0,\beta_0}(\F) = \nu_{\alpha_0,\beta_0}(\G)$ (which is thus
   zero).
   \item
   Say that \eqref{triangle} \emph{defines a $\lambda$-wall at the
   $\nu$-positive side of $(\alpha_0,\beta_0)$} if there is an open
   neighborhood $U$ of $(\alpha_0,\beta_0)$ such that, writing
   \begin{equation*}
      W = \{ (\alpha,\beta) \in U \;\mid\;
      \text{$\nu_{\alpha,\beta}(v) > 0$ and
      $\lambda_{\alpha,\beta,s}(\F) = \lambda_{\alpha,\beta,s}(\G)$} \}
   \end{equation*}
   the following holds: $(\alpha_0,\beta_0)$ is in the closure of $W$ and
   $\F,\E,\G$ are $\lambda_{\alpha,\beta,s}$-semistable objects in
   $\mathcal{A}^{\alpha,\beta}(X)$ for all $(\alpha,\beta)\in W$.
\end{itemize}

Note that the assumption that $\F,\E,\G$ are all in $\Coh^{\beta_0}(X)$
or in $\mathcal A^{\alpha,\beta}(X)$ implies that the triangle
\eqref{triangle} is a short exact sequence
\begin{equation*}
   0 \to \F \to \E \to \G \to 0
\end{equation*}
in that abelian category.

\begin{thm}[Schmidt]\label{thm:schmidt}
Let $\E$ be an object in $D^b(X)$ satisfying $\overline{\Delta}_H(\E)>0$
and let $(\alpha_0,\beta_0)\in\RR_{>0}\times \RR$ such that
$\nu_{\alpha_0,\beta_0}(\E) = 0$ and $\ch^{\beta_0}_1(\E) > 0$.
\begin{enumerate}
   \item If a triangle \eqref{triangle} defines a $\lambda$-wall on the
   $\nu$-positive side of $(\alpha_0,\beta_0)$, then it also defines a
   $\nu$-wall through $(\alpha_0,\beta_0)$. \item Suppose a triangle
   \eqref{triangle} defines a $\nu$-wall through $(\alpha_0,\beta_0)$
   and $\F$, $\G$ are $\nu_{\alpha_0,\beta_0}$-stable. Moreover let
   \begin{equation*}
      W = \{ (\alpha,\beta) \;\mid\; \text{$\nu_{\alpha,\beta}(v) > 0$ and
      $\lambda_{\alpha,\beta,s}(\F) = \lambda_{\alpha,\beta,s}(\G)$} \}
   \end{equation*}
   and suppose there are points $(\alpha,\beta)\in W$ arbitrarily close
   to $(\alpha_0,\beta_0)$ such that $\nu_{\alpha,\beta}(\F) > 0$ and
   $\nu_{\alpha,\beta}(\G) > 0$. Then \eqref{triangle} defines a
   $\lambda$-wall on the $\nu$-positive side of $(\alpha_0,\beta_0)$,
   namely $W$.
\end{enumerate}
\end{thm}

For the proof we refer to Schmidt \cite{Sch2020}: part (1) is Schmidt's
Theorem 6.1(1) and part (2) is the special case $n=1$ of Schmidt's
Theorem 6.1(4). To align the notation, in part (1) Schmidt's $\F,\E,\G$
are our $\F[1],\E[1],\G[1]$. To apply Theorem 6.1(1) these are required
to be $\lambda_{\alpha_0,\beta_0,s}$-semistable objects in $\mathcal
A^{\alpha_0,\beta_0}(X)$; this is ensured by Schmidt's Lemma 6.3.

To control how the set of stable objects changes as a
$\lambda$-wall is crossed, we take advantage of the fact that the
$\lambda$-walls we obtain are defined by short exact sequences with
stable sub- and quotient objects (in other words, only two
Jordan--H\"older factors on the wall) and apply:

\begin{pro}\label{prop:crossing}
Suppose $\F$ and $\G$ are $\lambda_{\alpha,\beta,s}$-stable objects in
$\mathcal A^{\alpha,\beta}(X)$. Then there is a neighborhood $U$ of
$(\alpha,\beta)$ such that for all $(\alpha',\beta')\in U$ and all
nonsplit extensions
\begin{equation*}
   0 \to \F \to \E \to \G \to 0
\end{equation*}
the object $\E$ is $\lambda_{\alpha',\beta',s}$-stable if and only if
$\lambda_{\alpha',\beta',s}(\F) < \lambda_{\alpha',\beta',s}(\G)$.
\end{pro}

This result is stated and proved (for arbitrary Bridgeland stability
conditions) in Schmidt \cite[Lemma 3.11]{Sch2020}, and credited there
also to Bayer--Macr\`i \cite[Lemma 5.9]{BM2011}.


\section{Wall and chamber structure}\label{sec:walls}

The starting point for the entire discussion that follows is a simple
minded observation. Namely, let $V\subset \PP^3$ be a plane and let $Y$
be the union of a conic in $V$ and a point $P$ also in $V$. Then there
is a short exact sequence
\begin{equation}\label{eq:extension1}
   0 \to \OO_{\PP^3}(-1) \to \I_Y \to \I_{P/V}(-2) \to 0
\end{equation}
(read $\OO_{\PP^3}(-1)$ as the ideal of $V$ and $\I_{P/V}(-2)$ as the
relative ideal of $Y\subset V$). If we instead let $Y$ be the union of a
conic in $V$ and a point $P$ outside of $V$ then there is a short exact
sequence
\begin{equation}\label{eq:extension2}
   0 \to \I_P(-1) \to \I_Y \to \OO_V(-2) \to 0
\end{equation}
(read $\I_P(-1)$ as the ideal of $\{P\} \cup V$ and $\OO_V(-2)$ as the
relative ideal of a conic in $V$). The claim is that in a certain region
of the stability manifold of $\PP^3$, there are exactly two walls with
respect to the Chern character $\ch(\I_Y) = (1,0,-2,2)$, and they are
defined precisely by the two pairs of sub and quotient objects appearing
in the short exact sequences \eqref{eq:extension1} and
\eqref{eq:extension2}.

Mimicking Schmidt's work for twisted cubics (and their deformations), we
argue via tilt stability. Since $\nu_{\alpha,\beta}$-stability only
involves Chern classes of codimension at least one, and the above two
short exact sequences are indistinguishable in codimension one, they
give rise to one and the same wall in the tilt stability parameter
space. Making this precise is the content of Section
\ref{sec:walls:tiltstability}. Moving on to
$\lambda_{\alpha,\beta,s}$-stability, we apply Schmidt's method to see
that the single $\nu_{\alpha,\beta}$-wall ``sprouts'' two distinct
$\lambda_{\alpha,\beta,s}$-walls corresponding to \eqref{eq:extension1}
and \eqref{eq:extension2}. This is carried out in Section
\ref{sec:walls:brdgstability}.

\subsection{$\nu$-stability wall}\label{sec:walls:tiltstability}

Throughout we specialize to $X=\PP^3$ with $H$ a (hyper)plane. Let
$v=(1,0,-2,2)$ be the Chern character of ideal sheaves $\I_Y$
of subschemes $Y \in \hilb^{2m+2}(\PP^3)$. 

For $(\alpha,\beta)\in \RR_{>0}\times\RR$ we have the tilted abelian
category $\Coh^{\beta}(\PP^3)$ and the slope function
$\nu_{\alpha,\beta}$. We concentrate on the region $\beta<0$, in which
any ideal $\I_Y$ of a subscheme $Y\subset \PP^3$ of dimension $\le 1$
satisfies 
\begin{equation*}
   \mu_\beta(\I_Y) = \frac{c_1(\I_Y)}{\rk(\I_Y)} - \beta = - \beta > 0.
\end{equation*}
As $\I_Y$ is $\mu$-stable also $\mu_\beta(\G) >0$ for every quotient
$\I_Y\twoheadrightarrow \G$ and so $\I_Y \in \mathcal T_{\beta}$. In
particular $\I_Y \in \Coh^{\beta}(\PP^3)$.

We begin by establishing that there is exactly one tilt-stability wall
in the region $\beta<0$. The result as well as the argument is analogous
to the analysis for twisted cubics by Schmidt \cite[Theorem
5.3]{Sch2020}, except that twisted cubics come with a second wall that
destabilizes all objects --- for our skew lines there is no such final
wall.

\begin{figure}
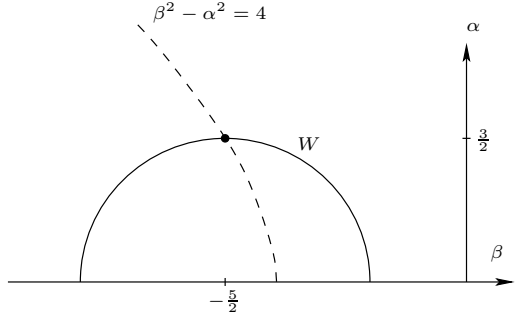

\centering
\begin{minipage}{0.7\textwidth}
\centering
{\pgfkeys{/pgf/fpu/.try=false}%
\ifx\XFigwidth\undefined\dimen1=0pt\else\dimen1\XFigwidth\fi
\divide\dimen1 by 3174
\ifx\XFigheight\undefined\dimen3=0pt\else\dimen3\XFigheight\fi
\divide\dimen3 by 1957
\ifdim\dimen1=0pt\ifdim\dimen3=0pt\dimen1=3946sp\dimen3\dimen1
  \else\dimen1\dimen3\fi\else\ifdim\dimen3=0pt\dimen3\dimen1\fi\fi
\tikzpicture[x=+\dimen1, y=+\dimen3]
{\ifx\XFigu\undefined\catcode`\@11
\def\temp{\alloc@1\dimen\dimendef\insc@unt}\temp\XFigu\catcode`\@12\fi}
\XFigu3946sp
\ifdim\XFigu<0pt\XFigu-\XFigu\fi
\pgfdeclarearrow{
  name = xfiga2,
  parameters = {
    \the\pgfarrowlinewidth \the\pgfarrowlength \the\pgfarrowwidth\ifpgfarrowopen o\fi},
  defaults = {
	  line width=+7.5\XFigu, length=+120\XFigu, width=+60\XFigu},
  setup code = {
    \dimen7 2.6\pgfarrowlength\pgfmathveclen{\the\dimen7}{\the\pgfarrowwidth}
    \dimen7 2\pgfarrowwidth\pgfmathdivide{\pgfmathresult}{\the\dimen7}
    \dimen7 \pgfmathresult\pgfarrowlinewidth
    \pgfarrowssettipend{+\dimen7}
    \pgfarrowssetbackend{+-1.25\pgfarrowlength}
    \dimen9 -\pgfarrowlength\advance\dimen9 by-0.5\pgfarrowlinewidth
    \pgfarrowssetlineend{+\dimen9}
    \dimen9 -\pgfarrowlength\advance\dimen9 by-0.5\pgfarrowlinewidth
    \pgfarrowssetvisualbackend{+\dimen9}
    \pgfarrowshullpoint{+\dimen7}{+0pt}
    \pgfarrowsupperhullpoint{+-1.25\pgfarrowlength}{+0.5\pgfarrowwidth}
    \pgfarrowssavethe\pgfarrowlinewidth
    \pgfarrowssavethe\pgfarrowlength
    \pgfarrowssavethe\pgfarrowwidth
  },
  drawing code = {\pgfsetdash{}{+0pt}
    \ifdim\pgfarrowlinewidth=\pgflinewidth\else\pgfsetlinewidth{+\pgfarrowlinewidth}\fi
    \pgfpathmoveto{\pgfqpoint{-1.25\pgfarrowlength}{-0.5\pgfarrowwidth}}
    \pgfpathlineto{\pgfqpoint{0pt}{0pt}}
    \pgfpathlineto{\pgfqpoint{-1.25\pgfarrowlength}{0.5\pgfarrowwidth}}
    \pgfpathlineto{\pgfqpoint{-\pgfarrowlength}{0pt}}
    \pgfpathclose
    \ifpgfarrowopen\pgfusepathqstroke\else\pgfsetfillcolor{.}
	\ifdim\pgfarrowlinewidth>0pt\pgfusepathqfillstroke\else\pgfusepathqfill\fi\fi
  }
}
\clip(588,-2746) rectangle (3762,-789);
\tikzset{inner sep=+0pt, outer sep=+0pt}
\pgfsetlinewidth{+7.5\XFigu}
\pgfsetstrokecolor{black}
\draw (2850,-2550) arc[start angle=+-0.00, end angle=+180.00, radius=+900];
\pgfsetcolor{black}
\filldraw  (1950,-1650) circle [radius=+23];
\pgfsetarrows{[line width=7.5\XFigu, width=37\XFigu, length=73\XFigu]}
\pgfsetarrowsend{xfiga2}
\draw (3450,-2550)--(3450,-1050);
\draw (600,-2550)--(3750,-2550);
\pgfsetarrowsend{}
\draw (1950,-2525)--(1950,-2575);
\draw (3425,-1650)--(3475,-1650);
\pgfsetbeveljoin
\pgfsetdash{{+60\XFigu}{+60\XFigu}}{++0pt}
\draw (2269,-2549)--(2269,-2546)--(2269,-2539)--(2268,-2527)--(2267,-2512)--(2266,-2494)
  --(2265,-2475)--(2263,-2456)--(2261,-2438)--(2259,-2420)--(2256,-2403)--(2252,-2386)
  --(2249,-2368)--(2244,-2349)--(2240,-2333)--(2235,-2317)--(2230,-2299)--(2225,-2280)
  --(2219,-2260)--(2213,-2239)--(2206,-2217)--(2198,-2194)--(2191,-2171)--(2183,-2147)
  --(2175,-2124)--(2166,-2100)--(2158,-2078)--(2150,-2055)--(2142,-2034)--(2134,-2013)
  --(2127,-1993)--(2119,-1974)--(2112,-1957)--(2105,-1940)--(2098,-1923)--(2090,-1906)
  --(2083,-1888)--(2074,-1870)--(2066,-1852)--(2057,-1834)--(2047,-1815)--(2037,-1796)
  --(2027,-1776)--(2016,-1757)--(2005,-1737)--(1994,-1718)--(1982,-1698)--(1970,-1679)
  --(1958,-1659)--(1945,-1639)--(1932,-1619)--(1919,-1599)--(1907,-1582)--(1895,-1564)
  --(1882,-1545)--(1869,-1526)--(1855,-1506)--(1840,-1485)--(1825,-1464)--(1810,-1442)
  --(1794,-1420)--(1777,-1398)--(1761,-1375)--(1744,-1352)--(1727,-1330)--(1711,-1308)
  --(1694,-1286)--(1678,-1265)--(1663,-1244)--(1648,-1225)--(1633,-1206)--(1619,-1188)
  --(1606,-1170)--(1593,-1154)--(1581,-1139)--(1569,-1124)--(1552,-1103)--(1537,-1084)
  --(1521,-1066)--(1506,-1048)--(1491,-1031)--(1476,-1013)--(1460,-995)--(1443,-977)
  --(1426,-959)--(1410,-942)--(1396,-927)--(1384,-915)--(1376,-906)--(1371,-901)
  --(1369,-899);
\pgftext[base,left,at=\pgfqpointxy{3600}{-2400}] {\fontsize{7}{8.4}\usefont{T1}{ptm}{m}{n}$\beta$}
\pgftext[base,left,at=\pgfqpointxy{3450}{-975}] {\fontsize{7}{8.4}\usefont{T1}{ptm}{m}{n}$\alpha$}
\pgftext[base,left,at=\pgfqpointxy{2400}{-1725}] {\fontsize{7}{8.4}\usefont{T1}{ptm}{m}{n}$W$}
\pgftext[base,left,at=\pgfqpointxy{1844}{-2700}] {\fontsize{7}{8.4}\usefont{T1}{ptm}{m}{n}$-\tfrac{5}{2}$}
\pgftext[base,left,at=\pgfqpointxy{3507}{-1690}] {\fontsize{7}{8.4}\usefont{T1}{ptm}{m}{n}$\tfrac{3}{2}$}
\pgftext[base,left,at=\pgfqpointxy{1500}{-900}] {\fontsize{7}{8.4}\usefont{T1}{ptm}{m}{n}$\beta^2-\alpha^2=4$}
\endtikzpicture}%
\caption{The unique semicircular wall $W$ (solid) for $\nu$-stability
and the hyperbola (dashed) from Equation \ref{eq:hyperbola}.}\label{fig:nuwalls}
\end{minipage}
\end{figure}

\begin{pro}\label{prop:tiltwall}
There is exactly one tilt-stability wall for objects with Chern
character $v = (1,0,-2,2)$ in the region $\beta<0$: it is the semicircle
\begin{equation*}
   W\colon\quad \alpha^2 + (\beta + \tfrac{5}{2})^2 = (\tfrac{3}{2})^2.
\end{equation*}
The wall is defined by exactly the unordered pairs of the following two
types:
\begin{enumerate}
   \item[(1)] $\left\{\I_P(-1), \OO_V(-2)\right\}$, where
   $V\subset \PP^3$ is a plane and $P\in V$, and
   \item[(2)] $\left\{\OO_{\PP^3}(-1), \I_{P/V}(-2)\right\}$, where
   $V\subset\PP^3$ is a plane and $P\not\in V$.
\end{enumerate}
Moreover, the four sheaves figuring in the above unordered pairs are
$\nu_{\alpha,\beta}$-stable objects in $\Coh^{\beta}(\PP^3)$ for all
$(\alpha,\beta)$ on $W$.
\end{pro}

The wall $W$ and the hyperbola $\nu_{\alpha,\beta}(v)=0$, intersecting
at $(\alpha,\beta)=(3/2, -5/2)$, are shown in Figure \ref{fig:nuwalls}.
Note that we visualize the $\alpha$-axis as the vertical one.

We first prove the final claim in the proposition. Here is a slightly
more general statement:

\begin{lem}\label{lem:tiltstability}
\leavevmode
\begin{enumerate}
   \item Let $Z\subset \PP^3$ be a finite, possibly empty subscheme.
   Then the ideal sheaf $\I_Z(-1)$ is a $\nu_{\alpha,\beta}$-stable
   object in $\Coh^{\beta}(\PP^3)$ for all $\alpha>0$ and $\beta<-1$.
   \item Let $V\subset \PP^3$ be a plane and $Z\subset V$ be a finite,
   possibly empty subscheme. Then the relative ideal sheaf
   $\I_{Z/V}(-2)$ is a $\nu_{\alpha,\beta}$-stable object in
   $\Coh^{\beta}(\PP^3)$ for all $(\alpha,\beta)\in \RR_{>0}\times \RR$
   such that
   \begin{equation*}
      \alpha^2 + (\beta + \tfrac{5}{2})^2 > (\tfrac{1}{2})^2.
   \end{equation*}
\end{enumerate}
\end{lem}

\begin{rem}
The condition on $(\alpha,\beta)$ in part (2) is necessary because of a
wall for $\I_{Z/V}(-2)$. For simplicity let $Z$ be empty. There is a
short exact sequence of coherent sheaves
\begin{equation*}
   0 \to \OO_{\PP^3}(-3) \to \OO_{\PP^3}(-2) \to \OO_V(-2) \to 0
\end{equation*}
which yields a short exact sequence
\begin{equation*}
   0 \to \OO_{\PP^3}(-2) \to \OO_V(-2) \to \OO_{\PP^3}(-3)[1] \to 0
\end{equation*}
in $\Coh^\beta(\PP^3)$ when $-3 < \beta < -2$. The condition
$\nu_{\alpha,\beta}(\OO_{\PP^3}(-2)) < \nu_{\alpha,\beta}(\OO_V(-2))$ is
exactly the inequality in (2).
\end{rem}

\begin{proof}[Proof of Lemma \ref{lem:tiltstability}]
The sheaf $\I_Z(-1)$ is $\mu$-stable and satisfies
$\mu_\beta(\I_Z(-1)) = -1 - \beta$. For all $\beta<-1$ it is thus an
object in $\mathcal T_\beta$ and so also in $\Coh^\beta(\PP^3)$. Since
$\I_{Z/V}(-2)$ is a torsion sheaf it too belongs to $\mathcal T_\beta$
and so to $\Coh^\beta(\PP^3)$, for all $\beta$.

We reduce to the situation $Z=\emptyset$. First consider $\I_Z(-1)$ and
assume $\beta<-1$. Note that $\I_Z(-1)$ is a subobject of
$\OO_{\PP^3}(-1)$ also in $\Coh^\beta(\PP^3)$ since the torsion sheaf
$\OO_Z$ belongs to that category and hence
\begin{equation*}
   0 \to \I_Z(-1) \to \OO_{\PP^3}(-1) \to \OO_Z \to 0
\end{equation*}
is a short exact sequence in $\Coh^\beta(\PP^3)$. Suppose
$\OO_{\PP^3}(-1)$ is $\nu_{\alpha,\beta}$-stable. Let $\F\subset
\I_Z(-1)$ be a proper nonzero subobject in $\Coh^\beta(\PP^3)$ with
quotient $\G$. View $\F$ also as a subobject of $\OO_{\PP^3}(-1)$, with
quotient $\G'$. Then $\nu_{\alpha,\beta}$ cannot distinguish between
$\G$ and $\G'$. Thus if $\OO_{\PP^3}(-1)$ is $\nu_{\alpha,\beta}$-stable
then
\begin{equation*}
   \nu_{\alpha,\beta}(\F) < \nu_{\alpha,\beta}(\G') = \nu_{\alpha,\beta}(\G)
\end{equation*}
and so $\I_Z(-1)$ is $\nu_{\alpha,\beta}$-stable as well. The reduction
from $\I_{Z/V}(-2)$ to $\OO_V(-2)$ is completely analogous.

$\nu_{\alpha,\beta}$-stability of the line bundle $\OO_{\PP^3}(-1)$ is a
consequence of $\overline{\Delta}_H(\OO_{\PP^3}(-1)) = 0$, by
\cite[Proposition 7.4.1]{BMT2014}.

The main task is to establish $\nu_{\alpha,\beta}$-stability of
$\OO_V(-2)$ in the region defined in part (2). By point (3) of \cite
[Theorem 3.3]{Sch2020}, the ray $\beta=-\frac{5}{2}$ intersects all
potential semicircular $\nu$-walls for $\ch(\OO_V(-2))$ at their top
point, meaning they must be centered at $(0,-\frac{5}{2})$. All such
semicircles of radius bigger than $\frac{1}{2}$ will intersect the ray
$\beta=-2$ (as well as $\beta=-3$). Thus it suffices to prove that
$\OO_V(-2)$ is $\nu_{\alpha,\beta}$-stable for all $\alpha>0$ and all
\emph{integers} $\beta$.

For such $(\alpha,\beta)$, suppose
\begin{equation*}
   0 \to \F \to \OO_V(-2) \to \G \to 0
\end{equation*}
is a short exact sequence in $\Coh^\beta(\PP^3)$ with $\F\ne 0$. We
claim that $\ch_1^\beta(\G) = 0$. This yields the result, since then
$\nu_{\alpha,\beta}(\G) = \infty$ and so $\OO_V(-2)$ is
$\nu_{\alpha,\beta}$-stable.

Let $r_\F = H^3\ch_0(\F)$ and $c_\F = H^2\ch_1(\F)$, i.e.\ the rank and
first Chern class considered as integers. Also let $r_\G = H^3\ch_0(\G)$
and $c_\G = H^2\ch_1(\G)$. By the short exact sequence we have
\begin{equation*}
   r_\F + r_\G = 0 \quad\text{and}\quad c_\F + c_\G = 1.
\end{equation*}

The induced long exact cohomology sequence of sheaves shows that
$\coho^{-1}(\F) = 0$, so from the short exact sequence
\begin{equation*}
   0 \to \coho^{-1}(\F)[1]  \to \F \to \coho^{0}(\F) \to 0
\end{equation*}
we see that $\F\cong \coho^{0}(\F)$ is a coherent sheaf in
$\mathcal{T}_\beta$. The remaining long exact sequence is
\begin{equation*}
   0 \to \coho^{-1}(\G) \to \F \to \OO_V(-2) \to \coho^0(\G) \to 0
\end{equation*}
The map into $\F$ cannot be an isomorphism, since $\coho^{-1}(\G)$ is in
$\mathcal T^\perp_\beta$ and $\F$ is in $\mathcal T_\beta$ and is
nonzero by assumption. Therefore the map in the middle is nonzero and so
the rightmost sheaf $\coho^0(\G)$ is a proper quotient of $\OO_V(-2)$
and so is a torsion sheaf supported in dimension $\le 1$. Thus only
$\coho^{-1}(\G)$ contributes to $r_\G$ and $c_\G$.

Suppose $r_\F\ne 0$. As $\F \in \mathcal{T}_{\beta}$ and $\coho^{-1}(\G)
\in \mathcal{T}^\perp_{\beta}$ we have 
\begin{equation*}
   \begin{cases}
   \mu_{\beta}(\F)>0\\
   \mu_{\beta}(\coho^{-1}(\G))\leq 0
   \end{cases} \implies
   \begin{cases}
   \dfrac{c_{\F}}{r_\F} -\beta > 0\\
   \dfrac{c_\G}{r_\G} -\beta \leq 0
   \end{cases} \implies
   \begin{cases}
   \dfrac{c_{\G}-1}{r_\G}-\beta > 0\\
   \dfrac{c_\G}{r_\G}-\beta \leq 0
   \end{cases}
\end{equation*}
and since $r_\G=-r_\F$ is negative we get $0 \le c_\G - \beta r_\G < 1$.
Since these are integers we must have $c_\G-\beta r_\G = 0$. Thus
\begin{equation*}
   \ch_1^{\beta}(\G) = (c_\G-\beta r_\G) H = 0
\end{equation*}
as claimed.

If on the other hand $r_\F=0$ then also $\coho^{-1}(\G)$ has rank zero
and hence must be zero as there are no torsion sheaves in
$\mathcal T^{\perp}_\beta$. Thus also $\G=\coho^0(\G)$ is a sheaf, with
vanishing rank and first Chern class. Again $\ch_1^{\beta}(\G)=0$ as
claimed. This completes the proof.
\end{proof}

By explicit computation (see \cite[Theorem 3.3]{Sch2020}), all numerical
tilt walls with respect to $v=(1,0,-2,2)$ in the region $\beta<0$ are
nested semicircles. More precisely, each is centered on the axis
$\alpha=0$ and has top point on the curve $\nu_{\alpha,\beta}(v)=0$,
that is the hyperbola
\begin{equation}\label{eq:hyperbola}
   \beta^2 - \alpha^2 = 4.
\end{equation}
In particular every tilt wall must intersect the ray $\beta=-2$.

We establish in the following lemma that there is at most one tilt
stability wall intersecting the ray $\beta=-2$ for Chern character $v$
and $\beta<0$. We also give the possible Chern characters of sub- and
quotient objects that define it. This lemma is tightly analogous to
Schmidt \cite[Lemma 5.5]{Sch2020}. We use an asterisk $*$ to denote an
unspecified numerical value.
\begin{lem}\label{lem:possiblecherncharacters}
Let $\beta_0=-2$ and let $\alpha>0$ be arbitrary. Suppose there is a short
exact sequence
\begin{equation*}
   0 \to \F \to \E \to \G \to 0
\end{equation*}
of $\nu_{\alpha,\beta_0}$-semistable objects in $\Coh^{\beta_0}(\PP^3)$
with $\ch(\E) = (1,0,-2,*)$ and $\nu_{\alpha,\beta_0}(\F) =
\nu_{\alpha,\beta_0}(\G)$. Then
\begin{equation*}
   \ch^{\beta_0}(\F) = (1,1,\tfrac{1}{2},*)
   \quad\text{and}\quad
   \ch^{\beta_0}(\G) = (0,1,-\tfrac{1}{2},*)
\end{equation*}
or the other way around.
\end{lem}

\begin{proof}
Keep $\beta_0=-2$ throughout. We compute $\ch^{\beta_0}(\E) =
(1,2,0,*)$. Let $\ch^{\beta_0}(\F) = (r,c,d,*)$ with $r,c\in \ZZ$ and
$d\in \frac{1}{2}\ZZ$. Then $\ch^{\beta_0}(\G) = (1-r,2-c,-d,*)$.

Since the (very weak) stability function $Z^{\textrm{tilt}}$ sends
effective classes to the upper half plane $\mathbb H \cup \{0\}$ and
$Z^{\textrm{tilt}}(\E) = Z^{\textrm{tilt}}(\F) + Z^{\textrm{tilt}}(\G)$
we have
\begin{equation*}
   0 \le \Im Z^{\textrm{tilt}}(\F) \le \Im Z^{\textrm{tilt}}(\E).
\end{equation*}
Since $\Im Z^{\textrm{tilt}} = H\ch_1^{\beta_0}$ this gives $0 \le c \le 2$.

If $c=0$ then $\nu_{\alpha,\beta_0}(\F) = \infty$ and
$\nu_{\alpha,\beta_0}(\G) < \infty$, which is a contradiction. Similarly
if $c=2$ then $\nu_{\alpha,\beta_0}(\F) < \infty$ and
$\nu_{\alpha,\beta_0}(\G) = \infty$, again a contradiction. Therefore
$c=1$.

With $c=1$ we compute
\begin{equation*}
   \nu_{\alpha,\beta_0}(\F) = d - \tfrac{1}{2}\alpha^2r
   \quad\text{and}\quad
   \nu_{\alpha,\beta_0}(\G) = -d - \tfrac{1}{2}\alpha^2(1-r)
\end{equation*}
and so the condition $\nu_{\alpha,\beta_0}(\F) =
\nu_{\alpha,\beta_0}(\G)$ says
\begin{equation}\label{eq:positive}
   \alpha^2 = \frac{4d}{2r-1}
\end{equation}
so this expression must be strictly positive.

Suppose $r\ge 1$ and apply the Bogomolov inequality (Proposition
\ref{BGtilt}) to $\F$:
\begin{equation*}
   0 \le \overline{\Delta}_H(\F) = 1 - 2rd
   \quad\implies\quad d \le \frac{1}{2r}
\end{equation*}
When $r\ge 1$ this gives $d \le \frac{1}{2}$. On the other hand the
positivity of \eqref{eq:positive} gives $d>0$ and as $d$ is a half
integer this leaves only the possibility $d=\frac{1}{2}$ and $r=1$.

Similarly suppose $r\le 0$ and apply the Bogomolov inequality to $\G$:
\begin{equation*}
   0 \le \overline{\Delta}_H(\G) = 1 + 2(1-r)d
   \quad\implies\quad d \ge -\frac{1}{2(1-r)}
\end{equation*}
When $r\le 0$ this gives $d\ge -\tfrac{1}{2}$. On the other hand the
positivity of \eqref{eq:positive} gives $d<0$ and as $d$ is a half
integer this leaves only the possibility $d=-\frac{1}{2}$ and $r=0$.
\end{proof}

\begin{proof}[Proof of Proposition \ref{prop:tiltwall}]	
Assume there is a tilt stability wall for $v=(1,0,-2,2)$, i.e. there is
a short exact sequence 
\begin{equation*}
   0 \to \F \to \E \to \G \to 0
\end{equation*}
of $\nu_{\alpha,\beta}$-semistable objects in $\Coh^{\beta}(\PP^3)$ with
$\ch(\E) = (1,0,-2,2)$ and $\nu_{\alpha,\beta}(\F) =
\nu_{\alpha,\beta}(\G)$. As already pointed out, the same conditions
then hold for some $(\alpha,\beta_0)$ with $\beta_0=-2$. Then by Lemma
\ref{lem:possiblecherncharacters}, up to swapping $\F$ and $\G$, we have
\begin{align}
   \ch^{\beta_0}(\F) &= (1,1,\tfrac{1}{2},*)\label{eq:invariant1}\\
   \ch^{\beta_0}(\G) &= (0,1,-\tfrac{1}{2},*).\label{eq:invariant2}
\end{align}
Given any pair $\F,\G$ of such objects, write out the condition
$\nu_{\alpha,\beta}(\F) = \nu_{\alpha,\beta}(\G)$ on $(\alpha,\beta)$ to
obtain the equation for the wall in question; this yields the semicircle
as claimed. Thus we have proved that there is at most one tilt-wall and
found its equation.

A further result of Schmidt \cite[Lemma 5.4]{Sch2020} (which requires
$\beta$ to be integral, and so applies for $\beta_0=-2$) says that the
only $\nu_{\alpha,\beta_0}$-semistable objects in
$\Coh^{\beta_0}(\PP^3)$ with the invariants \eqref{eq:invariant1} and
\eqref{eq:invariant2} are
\begin{align*}
   \F &\cong \I_Z(-1)\\
   \G &\cong \I_{Z'/V}(-2)
\end{align*}
for a finite subscheme $Z\subset \PP^3$, a plane $V\subset \PP^3$ and a
finite subscheme $Z'\subset V$ (where $Z$ and $Z'$ are allowed to be
empty). Let $n$ and $n'$ denote the lengths of $Z$ and $Z'$,
respectively. Again for $\beta_0=-2$ we compute
\begin{align*}
   \ch^{\beta_0}_3(\F) &= \ch^{\beta_0}_3(\I_Z(-1)) = \tfrac{1}{6} - n \\
   \ch^{\beta_0}_3(\G) &= \ch^{\beta_0}_3(\I_{Z'/V}(-2)) = \tfrac{1}{6} - n'
\end{align*}
and moreover $\ch^{\beta_0}_3(\E) = - \frac{2}{3}$. Thus from
$\ch^{\beta_0}_3(\E) = \ch^{\beta_0}_3(\F) + \ch^{\beta_0}_3(\G)$ we
find
\begin{equation*}
   n + n' = 1
\end{equation*}
and so either $Z$ is empty and $Z'$ is a point, or $Z$ is a point and
$Z'$ is empty. This proves that only the two listed pairs of semistable
objects $\F,\G$ may occur in a short exact sequence defining the wall.

To finish the proof it only remains to show that both pairs of objects
listed do in fact realize the wall. By Lemma \ref{lem:tiltstability},
the sheaves $\I_Z(-1)$ and $\I_{Z'/V}(-2)$ are in $\Coh^\beta(\PP^3)$
and are $\nu_{\alpha, \beta}$-semistable (in fact $\nu_{\alpha,
\beta}$-stable) for all $(\alpha,\beta)$ on the semicircle. Also, the
ideal $\E=\I_Y$ of any $Y\in \hilb^{2m+2}(\PP^3)$ is an object in
$\Coh^{\beta}(\PP^3)$ (when $\beta < 0$) and since any ideal is
$\mu$-stable it is $\nu_{\alpha,\beta}$-stable for $\alpha\gg 0$ (by
Proposition \ref{largevoltilt}). Hence it is $\nu_{\alpha,\beta}$-stable
outside the semicircle and at least $\nu_{\alpha,\beta}$-semistable on
the semicircle. Thus, short exact sequences of the types
\eqref{eq:extension1} and \eqref{eq:extension2}  define the wall and we
are done.
\end{proof}

\subsection{$\lambda$-stability walls}\label{sec:walls:brdgstability}
Next we apply Schmidt's Theorem \ref{thm:schmidt} to the single
$\nu_{\alpha,\beta}$-wall found in Proposition \ref{prop:tiltwall}; this
yields two $\lambda_{\alpha,\beta,s}$-walls. 

We set up notation first: for $(\alpha,\beta,s) \in
\RR_{>0}\times\RR\times\RR_{>0}$ we have the doubly tilted category
$\mathcal{A}^{\alpha,\beta}(\PP^3)$ and the slope function
$\lambda_{\alpha,\beta,s}$. Once and for all we fix an arbitrary value
$s>0$ and view the $(\alpha,\beta)$-plane $\RR_{>0}\times\RR$ as
parametrizing both $\nu_{\alpha,\beta}$-stability and
$\lambda_{\alpha,\beta,s}$-stability; as before we restrict to
$\beta<0$. Walls and chambers are taken with respect to the Chern
character $v=(1,0,-2,2)$.

Write $P_v\subset\RR_{>0}\times\RR$ for the open subset defined by
$\nu_{\alpha,\beta}(v)>0$ and $\beta<0$; this is the region to the left
of the hyperbola \eqref{eq:hyperbola} in Figure \ref{fig:lambdawalls}.
Theorem \ref{thm:schmidt} addresses walls in $P_v$ close to the boundary
hyperbola.

\begin{pro}\label{prop:lambda-walls}
There are exactly two $\lambda_{\alpha,\beta,s}$-walls with respect to
$v=(1,0,-2,2)$ in $P_v$ whose closure intersect the hyperbola
\eqref{eq:hyperbola}. They are defined exactly by the two pairs of
objects listed in Proposition \ref{prop:tiltwall}.
\end{pro}

This means that the two walls are
\begin{equation}\label{eq:W1}
   W_1 = \{
   (\alpha,\beta) \in P_v \st
   \lambda_{\alpha,\beta,s}(\OO_{\PP^3}(-1)) = \lambda_{\alpha,\beta,s}(\I_{Q/V}(-2))
   \}
\end{equation}
and
\begin{equation}\label{eq:W2}
   W_2 = \{
   (\alpha,\beta) \in P_v \st
   \lambda_{\alpha,\beta,s}(\I_P(-1)) = \lambda_{\alpha,\beta,s}(\OO_V(-2))
   \}
\end{equation}
and the pair of objects defining each wall (close to
$(\alpha_0,\beta_0)$) is unique.

We refrain from writing out the (quartic) equations defining them. They
do depend on $s$, but independently of $s$ they both intersect the
hyperbola \eqref{eq:hyperbola} in
$(\alpha_0,\beta_0)=(\frac{3}{2},-\frac{5}{2})$ and as we will show in
the following proof, $W_1$ has negative slope at $(\alpha_0,\beta_0)$
whereas $W_2$ has positive slope there. Thus $W_1$ lies above $W_2$
($\alpha$ bigger) in the intersection between $P_v$ and a small open
neighborhood of $(\alpha_0,\beta_0)$.

\begin{proof}
We apply Schmidt's Theorem \ref{thm:schmidt}. Firstly, when $\ch(\E) =
v$ we have $\ch_1^\beta(\E) = v_1-\beta v_0 = -\beta > 0$ and
$\overline{\Delta}_H(v)=v_1^2-2v_0v_2=4 > 0$ so the theorem applies. The
first part of the Theorem says that any $\lambda$-wall in $P_v$, having
a point $(\alpha_0,\beta_0)$ with $\nu_{\alpha_0,\beta_0}(v)=0$ in its
closure, must be defined by one of the two pairs $(\F, \G)$ listed in
Proposition \ref{prop:tiltwall}. This leaves $W_1$ and $W_2$ as the only
candidates. Moreover the sub- and quotient objects $\F$ and $\G$
appearing are $\nu_{\alpha_0,\beta_0}$-stable by Lemma
\ref{lem:tiltstability}. Thus the second part of the theorem says that
conversely, $W_1$ and $W_2$ are indeed $\lambda$-walls, provided they
contain points $(\alpha,\beta)$ arbitrarily close to $(\alpha_0,\beta_0)
= (\frac{3}{2},-\frac{5}{2})$ such that $\nu_{\alpha,\beta}(\F)>0$ and
$\nu_{\alpha,\beta}(\G)>0$. It remains to check this last condition.

So let $(\F,\G)$ be one of the pairs $(\OO_{\PP ^3}(-1), \I_{P/V}(-2))$
or $(\I_P(-1), \OO_V(-2))$. The region $P_v$ is bounded by the hyperbola
$\nu_{\alpha,\beta}(v) = 0$ and implicit differentiation readily shows
that this has slope $\frac{d\alpha}{d\beta} = -5/3$ at
$(\alpha_0,\beta_0)$. Similarly $\nu_{\alpha,\beta}(\F)=0$ has slope
$-1$ and $\nu_{\alpha,\beta}(\G)=0$ is just the line $\beta=-5/2$, and
in each case $\nu_{\alpha,\beta}>0$ is the region to the left of these
boundary curves. Thus it suffices to show that our walls have slope
$>-1$ at $(\alpha_0,\beta_0)$. Now each wall $W_i$ is defined by the
condition $\lambda_{\alpha,\beta,s}(\F)=\lambda_{\alpha,\beta,s}(\G)$,
which is equivalent to  
\begin{equation}\label{eq:definingeq}
   (\Re Z(\F))(\Im Z(\G)) = (\Re Z(\G))(\Im Z(\F))
\end{equation}
where $Z=Z_{\alpha,\beta,s}$ is the stability function defined in
\eqref{eq:Z^s} from Section \ref{sec:prelim:constr}. Implicit
differentiation of this equation at $(\alpha_0,\beta_0)$ gives, after
some work, that $W_1$ has slope
\begin{equation}\label{eq:slopeW1}
   - \left(\frac{27s}{16} + 1\right)^{-1} \in (-1,0)	
\end{equation}
and $W_2$ has slope
\begin{equation}\label{eq:slopeW2}
   \left(\frac{27s}{4} + 1\right)^{-1} \in (0,1)
\end{equation}
both of which are $>-1$, and we are done.
\end{proof}

\begin{figure}
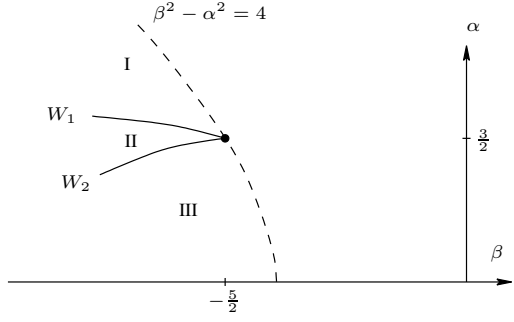

\centering
\begin{minipage}{0.7\textwidth}
\centering
{\pgfkeys{/pgf/fpu/.try=false}%
\ifx\XFigwidth\undefined\dimen1=0pt\else\dimen1\XFigwidth\fi
\divide\dimen1 by 3174
\ifx\XFigheight\undefined\dimen3=0pt\else\dimen3\XFigheight\fi
\divide\dimen3 by 1957
\ifdim\dimen1=0pt\ifdim\dimen3=0pt\dimen1=3946sp\dimen3\dimen1
  \else\dimen1\dimen3\fi\else\ifdim\dimen3=0pt\dimen3\dimen1\fi\fi
\tikzpicture[x=+\dimen1, y=+\dimen3]
{\ifx\XFigu\undefined\catcode`\@11
\def\temp{\alloc@1\dimen\dimendef\insc@unt}\temp\XFigu\catcode`\@12\fi}
\XFigu3946sp
\ifdim\XFigu<0pt\XFigu-\XFigu\fi
\pgfdeclarearrow{
  name = xfiga2,
  parameters = {
    \the\pgfarrowlinewidth \the\pgfarrowlength \the\pgfarrowwidth\ifpgfarrowopen o\fi},
  defaults = {
	  line width=+7.5\XFigu, length=+120\XFigu, width=+60\XFigu},
  setup code = {
    \dimen7 2.6\pgfarrowlength\pgfmathveclen{\the\dimen7}{\the\pgfarrowwidth}
    \dimen7 2\pgfarrowwidth\pgfmathdivide{\pgfmathresult}{\the\dimen7}
    \dimen7 \pgfmathresult\pgfarrowlinewidth
    \pgfarrowssettipend{+\dimen7}
    \pgfarrowssetbackend{+-1.25\pgfarrowlength}
    \dimen9 -\pgfarrowlength\advance\dimen9 by-0.5\pgfarrowlinewidth
    \pgfarrowssetlineend{+\dimen9}
    \dimen9 -\pgfarrowlength\advance\dimen9 by-0.5\pgfarrowlinewidth
    \pgfarrowssetvisualbackend{+\dimen9}
    \pgfarrowshullpoint{+\dimen7}{+0pt}
    \pgfarrowsupperhullpoint{+-1.25\pgfarrowlength}{+0.5\pgfarrowwidth}
    \pgfarrowssavethe\pgfarrowlinewidth
    \pgfarrowssavethe\pgfarrowlength
    \pgfarrowssavethe\pgfarrowwidth
  },
  drawing code = {\pgfsetdash{}{+0pt}
    \ifdim\pgfarrowlinewidth=\pgflinewidth\else\pgfsetlinewidth{+\pgfarrowlinewidth}\fi
    \pgfpathmoveto{\pgfqpoint{-1.25\pgfarrowlength}{-0.5\pgfarrowwidth}}
    \pgfpathlineto{\pgfqpoint{0pt}{0pt}}
    \pgfpathlineto{\pgfqpoint{-1.25\pgfarrowlength}{0.5\pgfarrowwidth}}
    \pgfpathlineto{\pgfqpoint{-\pgfarrowlength}{0pt}}
    \pgfpathclose
    \ifpgfarrowopen\pgfusepathqstroke\else\pgfsetfillcolor{.}
	\ifdim\pgfarrowlinewidth>0pt\pgfusepathqfillstroke\else\pgfusepathqfill\fi\fi
  }
}
\clip(588,-2746) rectangle (3762,-789);
\tikzset{inner sep=+0pt, outer sep=+0pt}
\pgfsetbeveljoin
\pgfsetlinewidth{+7.5\XFigu}
\pgfsetstrokecolor{black}
\draw (1951,-1649)--(1948,-1648)--(1940,-1646)--(1928,-1643)--(1910,-1638)--(1887,-1632)
  --(1861,-1625)--(1834,-1618)--(1805,-1610)--(1778,-1603)--(1752,-1597)--(1727,-1591)
  --(1704,-1585)--(1682,-1581)--(1661,-1576)--(1640,-1572)--(1619,-1568)--(1599,-1565)
  --(1579,-1562)--(1559,-1559)--(1539,-1556)--(1516,-1553)--(1492,-1549)--(1466,-1546)
  --(1438,-1543)--(1407,-1540)--(1374,-1536)--(1339,-1532)--(1302,-1529)--(1266,-1525)
  --(1230,-1521)--(1198,-1518)--(1171,-1516)--(1150,-1514)--(1136,-1512)--(1128,-1511)
  --(1124,-1511);
\draw (1951,-1649)--(1950,-1649)--(1947,-1650)--(1939,-1651)--(1925,-1653)--(1905,-1655)
  --(1880,-1659)--(1851,-1663)--(1821,-1668)--(1790,-1673)--(1760,-1678)--(1731,-1683)
  --(1705,-1688)--(1681,-1692)--(1658,-1697)--(1637,-1702)--(1616,-1707)--(1596,-1712)
  --(1576,-1718)--(1558,-1723)--(1540,-1729)--(1521,-1735)--(1501,-1742)--(1480,-1750)
  --(1458,-1759)--(1434,-1768)--(1407,-1779)--(1379,-1790)--(1350,-1802)--(1319,-1815)
  --(1288,-1828)--(1259,-1840)--(1233,-1852)--(1210,-1861)--(1193,-1869)--(1181,-1874)
  --(1174,-1877)--(1171,-1878);
\pgfsetfillcolor{black}
\pgftext[base,left,at=\pgfqpointxy{929}{-1964}] {\fontsize{7}{8.4}\usefont{T1}{ptm}{m}{n}$W_2$}
\pgftext[base,left,at=\pgfqpointxy{840}{-1538}] {\fontsize{7}{8.4}\usefont{T1}{ptm}{m}{n}$W_1$}
\pgftext[base,left,at=\pgfqpointxy{1327}{-1706}] {\fontsize{7}{8.4}\usefont{T1}{ptm}{m}{n}II}
\pgftext[base,left,at=\pgfqpointxy{1664}{-2136}] {\fontsize{7}{8.4}\usefont{T1}{ptm}{m}{n}III}
\pgftext[base,left,at=\pgfqpointxy{1322}{-1224}] {\fontsize{7}{8.4}\usefont{T1}{ptm}{m}{n}I}
\filldraw  (1950,-1650) circle [radius=+23];
\pgfsetarrows{[line width=7.5\XFigu, width=37\XFigu, length=73\XFigu]}
\pgfsetarrowsend{xfiga2}
\draw (3450,-2550)--(3450,-1050);
\draw (600,-2550)--(3750,-2550);
\pgfsetarrowsend{}
\draw (1950,-2525)--(1950,-2575);
\draw (3425,-1650)--(3475,-1650);
\pgfsetdash{{+60\XFigu}{+60\XFigu}}{++0pt}
\draw (2269,-2549)--(2269,-2546)--(2269,-2539)--(2268,-2527)--(2267,-2512)--(2266,-2494)
  --(2265,-2475)--(2263,-2456)--(2261,-2438)--(2259,-2420)--(2256,-2403)--(2252,-2386)
  --(2249,-2368)--(2244,-2349)--(2240,-2333)--(2235,-2317)--(2230,-2299)--(2225,-2280)
  --(2219,-2260)--(2213,-2239)--(2206,-2217)--(2198,-2194)--(2191,-2171)--(2183,-2147)
  --(2175,-2124)--(2166,-2100)--(2158,-2078)--(2150,-2055)--(2142,-2034)--(2134,-2013)
  --(2127,-1993)--(2119,-1974)--(2112,-1957)--(2105,-1940)--(2098,-1923)--(2090,-1906)
  --(2083,-1888)--(2074,-1870)--(2066,-1852)--(2057,-1834)--(2047,-1815)--(2037,-1796)
  --(2027,-1776)--(2016,-1757)--(2005,-1737)--(1994,-1718)--(1982,-1698)--(1970,-1679)
  --(1958,-1659)--(1945,-1639)--(1932,-1619)--(1919,-1599)--(1907,-1582)--(1895,-1564)
  --(1882,-1545)--(1869,-1526)--(1855,-1506)--(1840,-1485)--(1825,-1464)--(1810,-1442)
  --(1794,-1420)--(1777,-1398)--(1761,-1375)--(1744,-1352)--(1727,-1330)--(1711,-1308)
  --(1694,-1286)--(1678,-1265)--(1663,-1244)--(1648,-1225)--(1633,-1206)--(1619,-1188)
  --(1606,-1170)--(1593,-1154)--(1581,-1139)--(1569,-1124)--(1552,-1103)--(1537,-1084)
  --(1521,-1066)--(1506,-1048)--(1491,-1031)--(1476,-1013)--(1460,-995)--(1443,-977)
  --(1426,-959)--(1410,-942)--(1396,-927)--(1384,-915)--(1376,-906)--(1371,-901)
  --(1369,-899);
\pgftext[base,left,at=\pgfqpointxy{3600}{-2400}] {\fontsize{7}{8.4}\usefont{T1}{ptm}{m}{n}$\beta$}
\pgftext[base,left,at=\pgfqpointxy{3450}{-975}] {\fontsize{7}{8.4}\usefont{T1}{ptm}{m}{n}$\alpha$}
\pgftext[base,left,at=\pgfqpointxy{1844}{-2700}] {\fontsize{7}{8.4}\usefont{T1}{ptm}{m}{n}$-\tfrac{5}{2}$}
\pgftext[base,left,at=\pgfqpointxy{3507}{-1690}] {\fontsize{7}{8.4}\usefont{T1}{ptm}{m}{n}$\tfrac{3}{2}$}
\pgftext[base,left,at=\pgfqpointxy{1500}{-900}] {\fontsize{7}{8.4}\usefont{T1}{ptm}{m}{n}$\beta^2-\alpha^2=4$}
\endtikzpicture}%
\caption{The two walls $W_1$ and $W_2$ (solid) for $\lambda$-stability
separating three chambers, together with the hyperbola (dashed) from
Equation \ref{eq:hyperbola}.}\label{fig:lambdawalls}
\end{minipage}
\end{figure}

We are now in position to prove Theorem \ref{thm:main1}. By Proposition
\ref{prop:lambda-walls} there exists an open (connected) neighborhood $N
\subset \RR_{>0}\times \RR$ around the $\beta<0$ branch of the hyperbola
\eqref{eq:hyperbola}, such that the only $\lambda$-walls in $N\cap P_v$
are $W_1$ and $W_2$, defined in \eqref{eq:W1} and \eqref{eq:W2}.
Moreover it follows from the slopes \eqref{eq:slopeW1} and
\eqref{eq:slopeW2} that, after shrinking $N$ further if necessary, $W_1$
lies above ($\alpha$ bigger) $W_2$ throughout $N\cap P_v$. Thus the two
walls separate $N$ into three chambers, which we label I, II and III in
order of decreasing $\alpha$. 

\begin{proof}[Proof of Theorem \ref{thm:main1} (I)]
By Proposition \ref{prop:tiltwall} there is a single semicircular wall
(in the region $\beta<0$) for $\nu$-stability. It follows from Theorem
\ref{thm:refine} that the class of $\lambda_{\alpha,\beta,s}$-stable
objects in $\mathcal A^{\alpha,\beta}(\PP^3)$ for $(\alpha,\beta)$ in
chamber I coincides with the class of $\nu_{\alpha,\beta}$-stable
objects in $\Coh^{\beta}(\PP^3)$ for $(\alpha,\beta)$ outside the single
$\nu$-wall (up to shrinking $N$ even further if necessary).

Moreover, for $\alpha$ sufficiently big, the $\nu$-stable objects in
$\Coh^{\beta}(\PP^3)$ are exactly the $\mu$-stable coherent sheaves
(Proposition \ref{largevoltilt}). For Chern character $v=(1,0,-2,2)$
these are the ideals $\I_Y$ with $Y\in \hilb^{2m+2}(\PP^3)$.
\end{proof}

\begin{proof}[Proof of Theorem \ref{thm:main1} (II)]
Let $\E$ be $\lambda_{\alpha,\beta,s}$-stable for $(\alpha,\beta)$ in
chamber II. Since semistability is a closed property, $\E$ is semistable
on the wall $W_1$. If $\E$ is stable on the wall, then it is also stable
in chamber I hence it is an ideal sheaf in $\hilb^{2m+2}(\PP^3)$ by part
(I). Such an ideal remains stable on the wall if and only if it is not
an extension of the type \eqref{eq:extension1}, that is if and only if
it is the ideal of a nonplanar subscheme. This is case (II)(i) in the
Theorem.
	
If on the other hand $\E$ is stable in chamber II, but strictly
semistable on $W_1$, then by Proposition \ref{prop:lambda-walls} it is a
nonsplit extension of the pair
\begin{equation}\label{eq:firstpair}
   \big(\OO_{\PP^3}(-1),\I_{P/V}(-2)\big)
\end{equation} 
and we determine the direction of the extension (which object is the
subobject and which is the quotient) as follows: we claim that
\begin{equation}\label{eq:lambda-inequality}
   \lambda_{\alpha,\beta,s}(\I_{P/V}(-2))
   < \lambda_{\alpha,\beta,s}(\OO_{\PP^3}(-1))
\end{equation}
for all $(\alpha,\beta)$ in chamber II sufficiently close to
$(\alpha_0,\beta_0)$. Granted this, it follows that for $\E$ to be
stable in chamber II it must be a nonsplit extension as in case (II)(ii)
in the Theorem. Conversely it follows from Proposition
\ref{prop:crossing} that every such nonsplit extension is indeed stable
in chamber II. To verify
\eqref{eq:lambda-inequality} we let
\begin{equation}\label{eq:Phi}
   \Phi(\alpha,\beta)=\Re Z(\F)\Im Z(\G)-\Re Z(\G)\Im Z(\F)
\end{equation}
with $Z = Z_{\alpha,\beta,s}$,  $\F = \OO_{\PP^3}(-1)$ and $\G =
\I_{P/V}(-2)$. Thus $W_1$ is defined by $\Phi(\alpha,\beta)=0$ and
\eqref{eq:lambda-inequality} is equivalent to $\Phi(\alpha,\beta)<0$. It
thus suffices to check that the partial derivative of $\Phi$ with
respect to $\alpha$ is positive at $(\alpha_0,\beta_0)$. An explicit
computation yields in fact
\begin{equation*}
   \frac{\partial \Phi}{\partial \alpha}(\alpha_0,\beta_0) =2+\frac{27s}{8} > 0.
\end{equation*}

It remains only to show uniqueness of the nonsplit extensions
$\F_{P,V}$, that is
\begin{equation*}
   \dim\Ext^1_{\PP^3}(\OO_{\PP^3}(-1),\I_{P/V}(-2))=1.
\end{equation*}
But this space is $H^1(\I_{P/V}(-1))$, which is isomorphic to
$H^0(k(P)) = k$ via the short exact sequence
\begin{equation*}
   0 \to \I_{P/V}(-1) \to \OO_V(-1) \to k(P) \to 0.
\end{equation*}
\end{proof}

\begin{proof}[Proof of Theorem \ref{thm:main1} (III)]
Let $\E$ be $\lambda_{\alpha,\beta,s}$-stable for $(\alpha,\beta)$ in
chamber III. Since semistability is a closed property, $\E$ is
semistable on the wall $W_2$. If $\E$ is stable on $W_2$, then it is
stable in chamber II. This means two things: first, by part (II) of the
Theorem $\E$ is either an ideal sheaf of a nonplanar subscheme or a
nonsplit extension $\F_{P,V}$ as in case (II)(ii). Second, to remain
stable on $W_2$ the object $\E$ cannot be in a short exact sequence of
the type \eqref{eq:extension2} ruling out ideal sheaves of plane conics
union a point. Also the sheaves $\F_{P,V}$ sit in short exact sequences
of this type, as we show in Lemma \ref{lem:diagramF} below (the vertical
short exact sequence in the middle), and so are ruled out as well. Hence
$\E$ is an ideal sheaf of a disjoint pair of lines as claimed in
(III)(i). 

If on the other hand $\E$ is strictly semistable on $W_2$, then by
Proposition \ref{prop:lambda-walls} $\E$ is a nonsplit extension (in
either direction) of the pair
\begin{equation*}
   \big(\I_{P}(-1),\OO_V(-2)\big).
\end{equation*}
Now we claim that 
\begin{equation*}
   \lambda_{\alpha,\beta,s}(\OO_V(-2)<\lambda_{\alpha,\beta,s}(\I_{P}(-1))
\end{equation*}
for all $(\alpha,\beta)$ in chamber III sufficiently close to
$(\alpha_0,\beta_0)$. We prove this as in part II above, by partial
differentiation of $\Phi$ defined in \eqref{eq:Phi}, this time with $\F
= \I_{P}(-1)$ and $\G = \OO_V(-2)$. We find
\begin{equation*}
   \frac{\partial \Phi}{\partial \alpha}(\alpha_0,\beta_0)
   = \frac{1}{2}+\frac{27s}{8} > 0.
\end{equation*}
As before we conclude that $\E$ is a nonsplit extension as in (III)(ii)
and by Proposition \ref{prop:crossing} all such extensions are stable.

It remains to verify uniqueness of the extensions $\G_{P,V}$, i.e.
\begin{equation*}
   \dim\Ext^1_{\PP^3}(\I_{P}(-1),\OO_{V}(-2))=1
\end{equation*}
when $P\in V$. For this first apply $\Hom(-, \OO_V(-1))$ to the short
exact sequence
\begin{equation*}
   0 \to \I_P \to \OO_{\PP^3} \to k(P) \to 0
\end{equation*}
to obtain a long exact sequence which together with the vanishing of
$H^1(\OO_V(-1))$ and $H^2(\OO_V(-1))$ gives an isomorphism
\begin{equation*}
   \Ext^1_{\PP^3}(\I_{P},\OO_{V}(-1)) \iso \Ext^2(k(P), \OO_V(-1))
\end{equation*}
and ignoring twists, as these are not seen by $k(P)$, the right hand
side is Serre dual to $\Ext^1(\OO_V, k(P))$. This is one dimensional as
is seen by applying $\Hom(-, k(P))$ to the sequence
\begin{equation*}
   0 \to \OO_{\PP^3}(-1) \to \OO_{\PP^3} \to \OO_V \to 0.
\end{equation*}
\end{proof}

\subsection{The special sheaves}

Let $\F = \F_{P,V}$ and $\G = \G_{P,V}$ denote sheaves given by nonsplit
extensions of the form \eqref{eq:FpV} and \eqref{eq:GpV}, respectively.
The definition through (unique) nonsplit extensions is indirect and it
is useful to have alternative constructions available. We give such
constructions here and compute the spaces of first order infinitesimal
deformations.

\begin{lem}\label{lem:diagramF}
There is a commutative diagram with exact rows and columns as follows:
\begin{equation*}
   \begin{tikzcd}
   && 0 \ar[d] & 0 \ar[d] \\
   && \I_{P}(-1) \ar[d]\ar[r,equal] & \I_{P}(-1) \ar[d]\\
   0 \ar[r] &  \I_{P/V}(-2) \ar[d,equal]\ar[r] & \F \ar[d]\ar[r] & \OO_{\PP^3}(-1) \ar[d]\ar[r] & 0 \\
   0 \ar[r] & \I_{P/V}(-2) \ar[r] & \OO_V(-2) \ar[d]\ar[r] & k(P) \ar[d]\ar[r] & 0 \\
   && 0 & 0
   \end{tikzcd}
\end{equation*}
\end{lem}

\begin{proof}
Up to identifying the skyscraper sheaf $k(P)$ with any of its twists,
there are canonical short exact sequences as in the bottom row and the
rightmost column. The diagram can then be completed by letting $\F$ be
the fiber product as laid out by the square in the bottom right corner.
It remains only to verify that the middle row is nonsplit. But if it
were split the middle column twisted by $\OO_{\PP^3}(1)$ would be a
short exact sequence of the form
\begin{equation*}
   0 \to \I_P \to \I_{P/V}(-1) \oplus  \OO_{\PP^3} \to \OO_V(-1) \to 0.
\end{equation*}
Taking global sections this yields a contradictory left exact sequence
in which all terms vanish except $H^0(\OO_{\PP^3}) = k$.
\end{proof}

\begin{pro}\label{pro:dimExtF}
We have $\dim \Ext^1(\F, \F) = 11$.
\end{pro}

\begin{proof}
We will actually only prove that the dimension is at most $11$. The
opposite inequality may be shown by similar techniques, although it
follows from viewing $\Ext^1(\F, \F)$ as a Zariski tangent space to the
$11$-dimensional moduli space $\MII$ studied in the next section.

Apply $\Hom(-, \F)$ to the middle row in the diagram in Lemma
\ref{lem:diagramF}. This yields a long exact sequence
\begin{equation*}
   \cdots
   \to H^1(\F(1))
   \to \Ext^1(\F, \F)
   \to \Ext^1(\I_{P/V}(-2), \F)
   \to H^2(\F(1))
   \to \cdots
\end{equation*}
and from the middle column of the diagram we compute
$H^1(\F(1)) = H^2(\F(1)) = 0$. Thus we proceed to show that
$\dim \Ext^1(\I_{P/V}(-2), \F) \le 11$.

Apply $\Hom(\I_{P/V}(-2), -)$ to the middle row in the diagram.
This yields a long exact sequence:
\begin{equation*}
   \cdots
   \to \Ext^1(\I_{P/V}, \I_{P/V})
   \to \Ext^1(\I_{P/V}(-2), \F)
   \to \Ext^1(\I_{P/V}(-1), \OO_{\PP^3})
   \to \cdots
\end{equation*}
The space on the right is Serre dual to $H^2(\I_{P/V}(-5)) \iso
H^2(\OO_V(-5))$, which again on $V$ is Serre dual to $H^0(\OO_V(2))$.
This has dimension $6$. At least heuristically, the space on the left
should have dimension $5$, as it may be viewed as a tangent space to the
incidence variety $I\subset\PP^3\times\check\PP^3$ seen as a moduli
space for the sheaves $\I_{P/V}$. More directly we may apply
$\Hom(\I_{P/V}, -)$ to the Koszul complex on $V$
\begin{equation*}
   0 \to \OO_V(-2) \to \OO_V(-1)^{\oplus 2} \to \I_{P/V} \to 0
\end{equation*}
to obtain a long exact sequence
\begin{equation*}
   \cdots
   \to \underbrace{\Ext^1(\I_{P/V}, \OO_V(-1))^{\oplus 2}}_{2\cdot 2}
   \to \Ext^1(\I_{P/V}, \I_{P/V})
   \to \underbrace{\Ext^2(\I_{P/V}, \OO_V(-2))}_1
   \to \cdots
\end{equation*}
where the indicated dimensions may be computed by applying
$\Hom(\I_{P/V}(d), -)$ (for $d=1,2$) to the sequence
\begin{equation*}
   0 \to \OO_{\PP^3}(-1) \to \OO_{\PP^3} \to \OO_V \to 0.
\end{equation*}
We skip further details. It follows then that $\dim \Ext^1(\I_{P/V},
\I_{P/V}) \le 5$ and so $\dim \Ext^1(\I_{P/V}(-2), \F)$ is at most
$5+6=11$.
\end{proof}

\begin{lem}\label{lem:diagramG}
There is a commutative diagram with exact rows and columns as follows:
\begin{equation*}
   \begin{tikzcd}
   & & 0 \ar[d] & 0 \ar[d] \\
   & & \OO_{\PP^3}(-2) \ar[d] \ar[r, equal] & \OO_{\PP^3}(-2) \ar[d] \\
   0 \ar[r] & \OO_V(-2) \ar[d, equal] \ar[r] & \G \ar[d] \ar[r] & \I_P(-1) \ar[d] \ar[r] & 0 \\
   0 \ar[r] & \OO_V(-2) \ar[r] & \Omega^1_V \ar[r] \ar[d] & \I_{P/V}(-1) \ar[r] \ar[d] & 0 \\
   & & 0 & 0
   \end{tikzcd}
\end{equation*}
\end{lem}

\begin{proof}
From the Euler sequence
\begin{equation}\label{eq:euler}
   0 \to \Omega^1_V \to \OO_V(-1)^{\oplus 3} \to \OO_V \to 0
\end{equation}
on $V\iso\PP^2$ it follows that $\Omega^1_V(2)$ has a unique section (up
to scale) vanishing at $P\in V$. This leads to the short exact sequence
in the bottom row. Moreover there is a canonical short exact sequence as
in the rightmost column. The rest of the diagram can then be formed by
taking $\G$ to be the fiber product as laid out by the bottom
right square. It just remains to verify that the middle row is indeed
nonsplit. But if it were split the middle column would be a short exact
sequence of the form
\begin{equation*}
   0 \to \OO_{\PP^3}(-2) \to \OO_V(-2) \oplus \I_P(-1)
   \to \Omega^1_V \to 0.
\end{equation*}
This sequence implies that $H^1(\Omega^1_V(-1))$ is isomorphic to
$H^1(\I_P(-2))$, which is one dimensional. But the Euler sequence
shows that in fact $H^1(\Omega^1_V(-1)) = 0$.
\end{proof}

\begin{pro}\label{pro:dimExtG}
We have $\dim \Ext^1(\G, \G) = 8$.
\end{pro}

\begin{proof}
We will be using the short exact sequence
\begin{equation}\label{eq:G-Omega}
   0 \to \OO_{\PP^3}(-2) \to \G \to \Omega^1_V \to 0
\end{equation}
which sits as the middle column in Lemma \ref{lem:diagramG}. As
preparation we observe that all (dimensions of) $H^i(\Omega^1_V(d))$ may
be computed from the Euler sequence, and this enables us to compute
several $H^i(\G(d))$ from \eqref{eq:G-Omega}. We use these results
freely below without writing out further details.

Apply $\Hom(-, \G)$ to \eqref{eq:G-Omega} to produce a long exact
sequence
\begin{align*}
   0 &\to \underbrace{\Hom(\Omega^1_V, \G)}_0
   \to \underbrace{\Hom(\G, \G)}_1
   \to \underbrace{H^0(\G(2))}_4 \\
   &\to \Ext^1(\Omega_V^1, \G)
   \to \Ext^1(\G, \G)
   \to \underbrace{H^1(\G(2))}_0
   \to \cdots
\end{align*}
in which we have indicated some of the dimensions: $H^i(\G(2))$ are
computed from \eqref{eq:G-Omega} as sketched above, and since $\G$ is
simple we have $\Hom(\G, \G) = k$. For the same reason
\eqref{eq:G-Omega} is nonsplit, which implies
$\Hom(\Omega^1_V, \G) = 0$. It thus remains to see that the dimension of
$\Ext^1(\Omega^1_V, \G)$ is $11$.

Next apply $\Hom(-, \G)$ to the Euler sequence. This gives a long
exact sequence
\begin{align*}
   \cdots &\to \underbrace{\Hom(\Omega^1_V, \G)}_0
   \to \underbrace{\Ext^1(\OO_V, \G)}_1
   \to \underbrace{\Ext^1(\OO_V, \G(1))^{\oplus 3}}_{4\cdot 3}\\
   &\to \Ext^1(\Omega^1_V, \G)
   \to \underbrace{\Ext^2(\OO_V, \G)}_0
   \to \cdots
\end{align*}
where again dimensions have been indicated: the vanishing of
$\Hom(\Omega^1_V, \G)$ has already been noted, and there remain
several spaces of the form $\Ext^i(\OO_V, \G(d))$. These may be
computed from $H^i(\G(d))$ and the long exact sequence resulting
from applying $\Hom(-, \G(d))$ to
\begin{equation*}
   0 \to \OO_{\PP^3}(-1) \to \OO_{\PP^3} \to \OO_V \to 0.
\end{equation*}
It follows that $\dim \Ext^1(\Omega^1_V, \G) = 11$ and we are done.
\end{proof}


\section{Moduli spaces and universal families}\label{sec:modulispace}

By the classification of stable objects in chamber II, the moduli space
$\mathcal M^{\mathrm{II}}$ is at least obtained as a set from
$\hilb^{2m+2}(\PP^3) = \mathcal C \cup \mathcal S$ by just replacing the
divisor $E\subset \mathcal C$, parametrizing conics union a point inside
a plane, with the incidence variety $I$, parametrizing just pairs
$(P,V)$ of a point $P$ inside a plane $V$. Similarly, the moduli space
$\mathcal M^{\mathrm{III}}$ of stable objects in chamber III is obtained
from $\mathcal{M}^{\mathrm{II}}$ set-theoretically by removing
$\mathcal{M}^{\mathrm{II}}\setminus \mathcal{S}$ and replacing the
divisor $F \subset \mathcal S$, parametrizing pairs of incident lines
with a spatial embedded point at the intersection, with the incidence
variety $I$. We shall carry out each of these replacements as a
contraction, i.e.\ a blow-down, and prove that this indeed yields
$\mathcal M^{\mathrm{II}}$ and $\mathcal M^{\mathrm{III}}$, essentially
by writing down a universal family for each case.

\subsection{The contraction $\mathcal C \to \mathcal C'$}\label{sec:modulispace:contraction}

Recall that $\mathcal{C}$ is isomorphic to the blow-up of
$\PP^3\times\hilb^{2m+1}(\mathbb{P}^3)$ along the universal curve
$\mathcal{Z}$, where $\hilb^{2m+1}(\mathbb{P}^3)$ is the Hilbert scheme
of plane conics in $\PP ^3$ \cite{Lee2000}. The exceptional divisor $E'$
is comprised of plane conics with an embedded point. 

It is helpful to keep an eye at the following diagram
\begin{equation}\label{eq:Ccontr}
   \begin{tikzcd}
   \mathcal C \ar[r,"b"]
   & \PP^3\times\hilb^{2m+1}(\PP^3) \ar[d,"\pr_1"] \ar[r,"\pr_2"]
   & \hilb^{2m+1}(\PP^3) \ar[d,"\pi"]\\
   & \PP^3 & \check \PP^3
   \end{tikzcd}
\end{equation}
where $\pi$ sends a conic $C\in \hilb^{2m+1}(\PP^3)$ to the plane $V\in
\check\PP^3$ it spans and $b$ is the blowup along the universal family
$\mathcal{Z}$ of conics.

\begin{rem}\label{rem:localcontr}
It will sometimes be useful to resort to explicit computation in local
coordinates. For this let $U\subset \check\PP^3$ be the affine open
subset of planes $V\subset \PP^3$ with equation of the form
\begin{equation}\label{eq:Clocal1}
   x_3 = c_0x_0 + c_1x_1 + c_2x_2.
\end{equation}
Furthermore the $\PP^5$ of symmetric $3\times 3$ matrices $(s_{ij})$
parametrizes plane conics
\begin{equation}\label{eq:Clocal2}
   \sum_{0 \le i,j \le 2} s_{ij} x_ix_j = 0
\end{equation}
so that $\res{\hilb^{2m+1}(\PP^3)}{U} \iso \PP^5\times U$ with universal
family defined by the two equations \eqref{eq:Clocal1} and
\eqref{eq:Clocal2}. This is also the center for the blowup $b$, and we
note that it is nonsingular.
\end{rem}

\begin{lem}\label{lem:P5bundle}
$\pi$ is a Zariski locally trivial $\PP^5$-bundle. More precisely, let
$I\subset\PP^3\times\check\PP^3$ be the incidence variety and let
\begin{equation*}
   \E = \pr_{2*}(\res{\pr_1^*\OO_{\PP^3}(2)}{I}).
\end{equation*}
Then $\E$ is locally free of rank $6$ and $\hilb^{2m+1}(\PP^3)\iso
\PP(\E^\vee)$ over $\check\PP^3$.
\end{lem}

Here $\PP(\E^\vee)$ denotes the projective bundle parametrizing
lines in the fibers of $\E$. Starting with the observation that
the fiber of $\E$ over $V\in\check\PP^3$ is $H^0(V, \OO_V(2))$
(note that $H^1(V, \OO_V(2))=0$, so base change in cohomology applies)
the Lemma is straight forward and we refrain from writing out details.

Now let $E\subset \mathcal C$ be the locus of planar $Y\in \mathcal C$.
The condition on a disjoint union $Y = C \cup \{P\}$ to be in $E$ is
just that $P$ is in the plane $V$ spanned by $C$. For a conic with an
embedded point the condition $Y\subset V$ also singles out the scheme
structure at the embedded point. View $E$ as a variety over the
incidence variety $I\subset\PP^3\times\check\PP^3$ via the morphism
$(\id_{\PP^3}\times\pi)\circ b$ (refer to Figure \ref{conics} for simple
illustrations of the types of elements in $E$ and $E'$. In fact, we find
it helpful for keeping track of the various divisors introduced in this
section.)

\begin{pro}\label{prop:negative-one}
$E$ is a Zariski locally trivial $\PP^5$-bundle over the incidence
variety $I\subset \PP^3\times\check\PP^3$. The restriction
$\res{\OO_{\mathcal C}(E)}{\PP^5}$ to a fiber is isomorphic to
$\OO_{\PP^5}(-1)$.
\end{pro}

\begin{figure}[!htb]\centering
\includegraphics[width=0.6\textwidth]{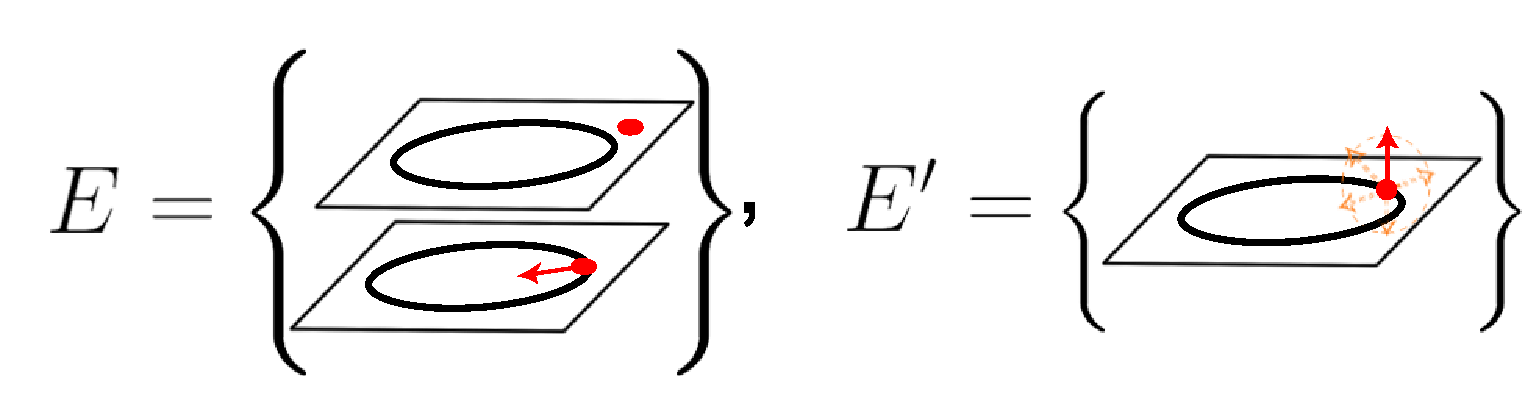}
\caption{A Circle represents a conic contained in a plane that is shown
as a parallelogram, and a red dot is a point, possibly embedded in the
conic. The arrow is the direction vector at an embedded point. Note that
in the left illustration, the arrow is strictly contained in the plane.} 
\label{conics}
\end{figure}

Before giving the proof, we harvest our application:

\begin{cor}\label{cor:Artincontraction}
There exist a smooth algebraic space $\mathcal C'$, a morphism
$\phi\colon \mathcal C \to \mathcal C'$ and a closed embedding $I\subset
\mathcal C'$, such that $\phi$ restricts to an isomorphism from
$\mathcal{C}\setminus E$ to $\mathcal{C'}\setminus I$ and to the given
projective bundle structure $E \to I$. Moreover $\phi$ is the blowup of
$\mathcal C'$ along $I$.
\end{cor}

It is well known that the condition verified in Proposition
\ref{prop:negative-one} implies the contractibility of $E/I$ in the
above sense. In the category of analytic spaces this is the Moishezon
\cite{Moi1967} or Fujiki--Nakano \cite{Nak1970, FN1971} criterion. In
the category of algebraic spaces the contractibility is due to Artin
\cite[Corollary 6.11]{Art1970}, although the statement there lacks the
identification with a blowup. Lascu \cite[Théorème 1]{Las1969} however
shows that once the contracted space $\mathcal C'$ as well as the image
$I \subset \mathcal C'$ of the contracted locus are both smooth, it does
follow that the contracting morphism is a blowup. Strictly speaking
Lascu works in the category of varieties, but our $\mathcal C'$ turns
out to be a variety anyway:

\begin{rem}
The algebraic space $\mathcal C'$ is in fact a projective variety. We
prove this in Section \ref{sec:mori} using Mori theory. As the arguments
there and in the present section are largely independent we separate the
statements.
\end{rem}

We also point out that the smooth contracted space $\mathcal C'$ is
unique once it exists: in general, suppose $\phi\colon X\to U$ and
$\psi\colon X\to V$ are proper birational morphisms between irreducible
separated algebraic spaces (say, of finite type over $k$) with $U$ and
$V$ normal. Moreover assume that $\phi(x_1)=\phi(x_2)$ if and only if
$\psi(x_1)=\psi(x_2)$. Let $\Gamma$ denote the image of
$(\phi,\psi)\colon X \to U\times V$. Then each of the projections from
$\Gamma$ to $U$ and $V$ is birational and bijective and hence an
isomorphism by Zariski's Main Theorem (for this in the language of
algebraic spaces we refer to the Stacks Project
\cite[\href{https://stacks.math.columbia.edu/tag/05W7}{Tag
05W7}]{stacks-project}).

\begin{proof}[Proof of Proposition \ref{prop:negative-one}.]
Consider the divisor
\begin{equation*}
   \overline{E}=\{(P,C) \in \PP^3 \times \hilb^{2m+1}(\mathbb{P}^3) \st
   \text{$P$ in the plane spanned by $C$} \}.
\end{equation*} 
It follows from Lemma \ref{lem:P5bundle} that
\begin{equation*}
   \PP^3 \times \hilb^{2m+1}(\mathbb{P}^3)
   \xrightarrow{id_{\PP^3}\times \pi} \PP^3 \times \check\PP^3
\end{equation*}
is a $\PP^5$-bundle, hence its restriction to $\big(id_{\PP^3}\times
\pi\big)^{-1}(I)=\overline{E}$ is a $\PP ^5$-bundle over $I \subset
\mathbb{P}^3 \times \check\PP^3$. 

Now, $E$ is the strict transform of $\overline{E}$, i.e.\ its blow-up
along $\mathcal Z \subset \overline{E}$. But this is a Cartier divisor,
since $\overline{E}$ is smooth, and so $E\iso\overline E$. This proves
the first claim.

Again using that $\overline E$ is smooth, its strict transform $E$
satisfies the linear equivalence
\begin{equation}\label{eq:total-strict}
   E = b^*(\overline E) - E'.
\end{equation}
The term  $b^*(\overline E) = b^*(\id_{\PP^3}\times\pi)^*(I)$ is a
pullback from the base of the $\PP^5$-bundle, so its restriction to any
fiber is trivial. Thus it suffices to see that $E'$ restricted to a
fiber $\PP^5$ is a hyperplane. Now the isomorphism $b\colon E
\iso\overline E$ identifies $E\cap E'\subset E$ with $\mathcal Z\subset
E$. In the local coordinates from Remark \ref{rem:localcontr} the
divisor $E$ is given by equation \eqref{eq:Clocal1} and $\mathcal Z$ is
given by the additional equation \eqref{eq:Clocal2}. Here $(s_{ij})$ are
the coordinates on the fiber $\PP^5$ and clearly \eqref{eq:Clocal2}
defines a hyperplane in each fiber --- it is the linear condition on the
space of plane conics given by passage through a given point.
\end{proof}

\begin{rem}
The locus $\mathcal C\cap \mathcal S$ consists of pairs of intersecting
lines with a \emph{spatial} embedded point at the intersection (and, as
degenerate cases, planar double lines with a spatial embedded point). On
the other hand $E$ consists only of planar objects, so $E$ is disjoint
from $\mathcal S$. Thus we may extend the contraction $\phi$ to a
morphism between algebraic spaces
\begin{equation*}
   (\phi\cup \id)\colon \mathcal C\cup \mathcal S \to \mathcal C'
   \cup \mathcal S
\end{equation*}
which is an isomorphism away from $E$ and restricts to the
$\PP^5$-bundle $E\to I$ as before.
\end{rem}

\subsection{Moduli in chamber II}\label{sec:familychamberII}

In this section we shall modify the universal family on the Hilbert
scheme $\hilb^{2m+2}(\PP^3) = \mathcal C \cup \mathcal S$ in such a way
that we replace its fibers over $E\subset \mathcal C$ with the objects
$\F_{P,V}$ in Theorem \ref{thm:main1}. This family induces a morphism
\begin{equation*}
   \mathcal C \cup \mathcal S \to \MII
\end{equation*}
and we conclude via uniqueness of normal (in this case smooth)
contractions that $\MII$ coincides with $\mathcal C'\cup \mathcal S$.

Here is the construction: let
\begin{equation*}
   \mathcal Y \subset \PP^3 \times \hilb^{2m+2}(\PP^3)
\end{equation*}
be the universal family and let
\begin{equation*}
   \mathcal V \subset \PP^3 \times E
\end{equation*}
be the $E$-flat family whose fiber $\mathcal V_\xi \subset \PP^3$ over a
point $\xi$ mapping to $(P,V)\in I$ is the plane $V$. Clearly $\mathcal
V$ can be written down as a pullback of the universal plane over
$\check\PP^3$. We view $\mathcal V$ as a closed subscheme of
$\PP^3\times\hilb^{2m+2}(\PP^3)$. Then our modified universal family is
the ideal sheaf $\I_{\mathcal Y\cup \mathcal V}$.

\begin{rem}
We emphasize the (to us, at least) unusual situation that $\I_{\mathcal
Y\cup \mathcal V}$ is the ideal of a very nonflat subscheme,
yet as we show below it is flat as a coherent sheaf. Its fibers over
points in $E$ are not ideals at all, but rather the objects $\F_{P,V}$.
\end{rem}

\begin{thm}\label{thm:modulichamberII}
As above let $\mathcal Y$ be the universal family over
$\hilb^{2m+2}(\PP^3)$ and $\mathcal V$ the family of planes in $\PP^3$
parametrized by $E$.
\begin{itemize}
   \item[(i)]
   $\I_{\mathcal Y \cup \mathcal V}$ is flat as a coherent sheaf over
   $\hilb^{2m+2}(\PP^3)$. Its fibers $\I_{\mathcal Y\cup\mathcal V}
   \tensor k(\xi)$ over $\xi\in\hilb^{2m+2}(\PP^3)$ are stable objects
   for stability conditions in chamber II.
   \item[(ii)]
   The morphism $\hilb^{2m+2}(\PP^3) \to \MII$ determined by
   $\I_{\mathcal Y\cup \mathcal V}$ induces an isomorphism $\mathcal C'
   \cup \mathcal S \iso \MII$.
\end{itemize}
\end{thm}

We begin by showing that the fibers $\I_{\mathcal Y\cup \mathcal V}
\tensor k(\xi)$ over $\xi\in E$ sit in a short exact sequence of the
type \eqref{eq:FpV}. The mechanism producing such a short exact sequence
is quite general. Note that when $Y = \mathcal Y_\xi$ is a conic with a
(possibly embedded) point $P$ in a plane $V = \mathcal V_\xi$, we have
$\I_{Y/V} \iso \I_{P/V}(-2)$ and $\I_V \iso \OO_{\PP^3}(-1)$.

\begin{lem}\label{lem:opposite-ses}
Let $X$ be a projective scheme, $\mathcal Y\subset X\times S$ an
$S$-flat subscheme, $E\subset S$ a Cartier divisor and $\mathcal
V\subset E\times S$ an $E$-flat subscheme such that $\mathcal Y \cap
(E\times S) \subset \mathcal V$. Let $\xi\in E$. Then there is a short
exact sequence
\begin{equation*}
   0 \to \I_{\mathcal Y_\xi/\mathcal V_\xi}
   \to \I_{\mathcal Y\cup\mathcal V} \tensor k(\xi)
   \to \I_{\mathcal V_\xi} \to 0.
\end{equation*}
In particular if $S$ is integral in a neighborhood of $E$ then
$\I_{\mathcal Y\cup\mathcal V}$ is flat over $S$.
\end{lem}

\begin{proof}
Observe that the last claim is implied by the first: outside of $E$ the
ideal $\I_{\mathcal Y\cup \mathcal V}$ agrees with $\I_{\mathcal Y}$,
which is flat. For $\xi\in E$ we have $\mathcal Y_\xi
\subset \mathcal V_\xi$ and so a short exact sequence
\begin{equation*}
   0 \to \I_{\mathcal V_\xi}
   \to \I_{\mathcal Y_\xi}
   \to \I_{\mathcal Y_\xi/\mathcal V_\xi} \to 0.
\end{equation*}
The short exact sequence in the statement has the same sub and quotient
objects in opposite roles, so the Hilbert polynomial of
$\I_{\mathcal Y\cup\mathcal V} \tensor k(\xi)$ agrees with that of
$\I_{\mathcal Y_\xi}$. Thus the Hilbert polynomial of the fibers of
$\I_{\mathcal Y\cup\mathcal V}$ is constant; over an integral base this
implies flatness.

Begin with the short exact sequence
\begin{equation*}
   0 \to \I_{\mathcal Y\cup\mathcal  V} \to \OO_{X\times S} \to
   \OO_{\mathcal Y\cup\mathcal  V} \to 0
\end{equation*}
and tensor with $\OO_{X\times E}$ to obtain the exact sequence
\begin{equation}\label{eq:tor-sequence}
   0 \to
   \TTor_1^{X\times S}(\OO_{\mathcal Y\cup \mathcal V}, \OO_{X\times E})
   \to \res{\I_{\mathcal Y\cup \mathcal V}}{E}
   \to \OO_{X\times E}
   \to \OO_{\mathcal V}
   \to 0.
\end{equation}
The kernel of the rightmost map is clearly the ideal $\I_{\mathcal
V} \subset \OO_{X\times E}$. To compute the Tor-sheaf on the left use
the short exact sequence
\begin{equation*}
   0 \to \OO_S(-E) \to \OO_S \to \OO_E \to 0.
\end{equation*}
Pull this back to $X\times S$ and tensor with $\OO_{\mathcal Y\cup
\mathcal V}$ to see that $\TTor_1^{X\times S}(\OO_{\mathcal Y\cup
\mathcal V}, \OO_{X\times E})$ is isomorphic to the kernel of the
homomorphism
\begin{equation*}
   \OO_{\mathcal Y\cup \mathcal V}(-\mathit{pr}_2^*E)
   \to \OO_{\mathcal Y\cup \mathcal V}
\end{equation*}
which locally is multiplication by an equation for $E$. Thus
\begin{equation*}
   \TTor_1^{X\times S}(\OO_{\mathcal Y\cup \mathcal V}, \OO_{X\times E})
   \iso \J(-\mathit{pr}_2^*E)
\end{equation*}
where $\J\subset \OO_{\mathcal Y\cup \mathcal V}$ is the ideal
locally consisting of elements annihilated by an equation for $E$.

We compute $\J$ in an open affine subset $\Spec A \subset X\times S$ in
which $\mathcal Y$ and $\mathcal V$ are given by ideals $I_{\mathcal Y}$
and $I_{\mathcal V}$ respectively and $f\in A$ is (the pullback of) a
local equation for $E$. Thus $\J$ corresponds to $(I_{\mathcal Y}\cap
I_{\mathcal V} : f)/(I_{\mathcal Y}\cap I_{\mathcal V})$. Now $f\in
I_{\mathcal V}$ since $\mathcal V\subset X\times E$. This implies that
for $g\in A$ the condition $fg\in I_{\mathcal Y}$ is equivalent to
$fg\in I_{\mathcal Y}\cap I_{\mathcal V}$ and so $(I_{\mathcal Y}\cap
I_{\mathcal V} : f) = (I_{\mathcal Y} : f)$. Moreover the latter equals
$I_{\mathcal Y}$, since $\mathcal Y$ is flat over $S$, so that
multiplication by the non-zero-divisor $f$ remains injective after
tensor product with $\OO_{\mathcal Y}$, that is $A/I_{\mathcal Y}$. Thus
$\J$ is locally
\begin{align*}
   (I_{\mathcal Y}\cap I_{\mathcal V} :f)/(I_{\mathcal Y}\cap I_{\mathcal V})
   &= I_{\mathcal Y}/(I_{\mathcal Y}\cap I_{\mathcal V})\\
   &\iso (I_{\mathcal Y}+I_{\mathcal V})/I_{\mathcal V}
   = I_{\mathcal{Y}_E}/I_{\mathcal V}
\end{align*}
where we write $\mathcal{Y}_E$ for the restriction $\mathcal Y \cap
(X\times E) = \mathcal Y \cap \mathcal V$. This shows
\begin{equation*}
   \TTor_1^{X\times S}(\OO_{\mathcal Y\cup \mathcal V}, \OO_{X\times E})
   \iso \I_{\mathcal{Y}_E/\mathcal{V}}(-\mathit{pr}_2^*E)
\end{equation*}
and \eqref{eq:tor-sequence} gives
the short exact sequence
\begin{equation*}
   0 \to \I_{\mathcal{Y}_E/\mathcal{V}}(-\mathit{pr}_2^*E)
   \to \I_{\mathcal Y\cup \mathcal V}\vline_E
   \to \I_{\mathcal V}
   \to 0
\end{equation*}
on $X\times E$. Finally restrict to the fiber over a point $\xi\in E$:
since $\mathcal{Y}_E$ and $\mathcal V$ are both $E$-flat this yields the
short exact sequence in the statement.
\end{proof}

Lemma \ref{lem:opposite-ses} does not guarantee that the short exact
sequence obtained is nonsplit. Showing this in the case at hand requires
some work. Our strategy is to exhibit a certain quotient sheaf
$\I_{\mathcal Y \cup \mathcal V}\tensor k(\xi) \onto \Q$ and check that
the split extension $\I_{\mathcal V_\xi} \oplus \I_{\mathcal Y_\xi /
\mathcal V_\xi}$ admits no surjection onto $\Q$. In fact $\Q =
\OO_V(-2)$ will work:

\begin{lem}\label{lem:nosurjectionO(-2)}
Let $V\subset\PP^3$ be a plane and $P\in V$ a point. Then there is no
surjection from $\OO_{\PP^3}(-1) \oplus \I_{P/V}(-2)$ onto $\OO_V(-2)$.
\end{lem}

\begin{proof}
Just note that $\Hom(\I_{P/V}(-2), \OO_V(-2)) = k$ is generated by the
(nonsurjective) inclusion, whereas
$\Hom(\OO_{\PP^3}(-1), \OO_V(-2)) = 0$.
\end{proof}

We will produce the required quotient sheaf by the following
construction, which depends on the choice of a tangent direction at
$\xi$ in $S$:

\begin{lem}\label{lem:quotientexists1}
With notation as in Lemma \ref{lem:opposite-ses}, let
$T = \Spec k[t]/(t^2)$ and let $T\subset S$ be a closed embedding such
that $T\cap E$ is the reduced point $\{\xi\}$. Let $Y = \mathcal Y_\xi$
and $V = \mathcal V_\xi$.

Define a subscheme $Y'\subset Y$ by the ideal 
\begin{equation*}
   (\I : t)/(t) \subset \OO_V
\end{equation*}
where $\I \subset \OO_{V \times T}$ is the ideal of
$\mathcal Y \cap (V \times T)$. Then there is a surjection
\begin{equation*}
   \I_{\mathcal Y\cup \mathcal V} \tensor k(\xi)
   \onto \I_{Y'/V}.
\end{equation*}
\end{lem}

\begin{rem}
Since $t^2=0$ we trivially have $t \in (\I:t)$. Since also $\I \subset
(\I : t)$ we furthermore have $Y'\subset Y$. If we extend $T$ to an
actual one parameter family of objects $Y_t$, we may think of $Y'$ as
the limit of $Y_t \cap V$ as $t\to 0$, in other words it is the part of
$Y$ that remains in $V$ as we deform along our chosen direction.
\end{rem}

\begin{proof}
Let $\mathcal Y_T \subset \PP^3\times T$ denote the restriction of
$\mathcal Y$ to $T$. We \emph{claim} that $\I_{Y'/V}$ is isomorphic to
the relative ideal of $\mathcal Y_T \cup (V\times \{\xi\})$ in
$\mathcal Y_T \cup (V\times T)$. Assuming this, there are surjections
\begin{equation*}
   \I_{\mathcal Y \cup \mathcal V}\tensor \OO_T
   \onto \I_{\mathcal Y_T \cup (V\times \{\xi\})}
   \onto \I_{Y'/V}
\end{equation*}
(the middle term is the ideal of $\res{(\mathcal Y\cup \mathcal V)}{T} =
\mathcal Y_T \cup (V \times\{\xi\})$ as a subscheme of $\PP^3\times T$).
Restriction to the fiber over $\xi$ gives the surjection in the
statement.

To prove the claim, we first observe that for any two subschemes $A$ and
$B$ of some ambient scheme, there is an isomorphism
\begin{equation*}
   \I_{(A\cap B)/A} \iso \I_{B/(A\cup B)}
\end{equation*}
between the relative ideal sheaves; this is the identity $(I+J)/I \iso
J/(I\cap J)$ between quotients of ideals. Apply this to
\begin{equation*}
   A = V \times T,\quad
   B = \mathcal Y_T \cup (V\times \{\xi\})
\end{equation*}
so that
\begin{align*}
   A \cup B &=
   \mathcal Y_T \cup (V \times T) \\
   A \cap B &=
   (\mathcal Y_T \cup (V \times \{\xi\})) \cap (V\times T)
   = (\mathcal Y \cap (V\times T)) \cup (V\times \{\xi\}).
\end{align*}
The \emph{claim} as stated thus says that  $\I_{Y'/V}$ is isomorphic to
$\I_{B/A\cup B}$, and we are free to replace the latter by
$\I_{(A\cap B)/A}$.

Next let $\Spec R$ be an affine open subset in $V$ and $I\subset
R[t]/(t^2)$ the ideal defining $\mathcal Y \cap (V\times T)$ there.
Locally the ideal $\I_{(A\cap B)/A}$ is then $I\cap (t) \subset
R[t]/(t^2)$. Now multiplication with $t$ is an isomorphism of
$R[t]/(t^2)$-modules
\begin{equation*}
   (I:t)/(t) \iso I\cap (t).
\end{equation*}
The left hand side is precisely $\I_{Y'/V}$ in the open subset
$\Spec R$.
\end{proof}

\begin{proof}[Proof of Theorem \ref{thm:modulichamberII} (i)]
By Lemma \ref{lem:opposite-ses} it suffices to show that the fiber of
$\I_{\mathcal Y \cup \mathcal V} \tensor k(\xi)$ over $\xi\in E$ is
isomorphic to $\F_{P,V}$. Moreover, for such $\xi$, the same Lemma
yields a short exact sequence
\begin{equation*}
   0 \to \I_{P/V}(-2) \to
   \I_{\mathcal Y \cup \mathcal V} \tensor k(\xi)
   \to \OO_{\PP^3}(-1) \to 0
\end{equation*}
and since $\F_{P,V}$ is the unique such nonsplit extension it is enough
to show that the above extension is nonsplit. In view of Lemma
\ref{lem:nosurjectionO(-2)} this follows once we can show the existence
of a surjection
\begin{equation*}
   \I_{\mathcal Z \cup \mathcal W} \tensor k(\xi) \onto \OO_V(-2).
\end{equation*}
For this it suffices, in the notation of Lemma
\ref{lem:quotientexists1}, to choose $T\subset \hilb^{2m+2}(\PP^3)$ such
that the subscheme $Y'\subset V$ is a conic.

\emph{Nondegenerate case.} First assume $Y = \mathcal Y_\xi$ is a
disjoint union $Y = C\cup\{P\}$ of a conic $C\subset V$ and a point
$P\in V$. Consider the one parameter family $Y_t = C \cup \{P_t\}$ in
which the conic part $C$ is fixed while the point $P_t$ travels along a
line intersecting $V$ in the point $P$. In suitable affine coordinates
we may take $V$ to be the plane $z=0$ in $\Aff^3 = \Spec k[x,y,z]$, the
point $P$ to be the origin and $C$ to be given by some quadric $q =
q(x,y)$ not vanishing at $P$. Let the one parameter family over $\Spec
k[t]$ consist of the union of $C$ with the point $P_t = (0,0,t)$. This
is given by the ideal
\begin{equation*}
   (q, z) \cap (x,y,z-t) = (xq, yq, (z-t)q, xz,yz,(z-t)z).
\end{equation*}
Now restrict to $T = \Spec k[t]/(t^2)$ and intersect the family with
$V\times T$. The resulting subscheme is defined by the ideal 
\begin{equation*}
   I = (xq, yq, (z-t)q, xz,yz,(z-t)z) + (z)
   = (xq, yq, t q, z)
\end{equation*}
and $(I:t)/(t) = (q, z)$ which defines $C\subset V$. Thus $Y' = C$ and
we are done.

\emph{Embedded point with nonsingular support.} Suppose $Y$ is a conic
$C\subset V$ with an embedded point supported at a point $P$ in which
$C$ is nonsingular, where the normal direction corresponding to the
embedded point is along $V$. Then take the one parameter family in which
$C$ and the supporting point $P$ is fixed and the embedded structure
varies in the $\PP^1$ of normal directions. In suitable affine
coordinates we may take $V$ to be the plane $z=0$ in $\Aff^3 = \Spec
k[x,y,z]$, $P$ to be the origin and $C$ given by a quadric $q = q(x,y)$
vanishing at $P$ and with, say, linear term $y$. Take the one parameter
family of $C$ with an embedded point given by
\begin{equation*}
   (q, z) \cap (x,y^2,z-ty) = (xq, yq, zq, z-tq)
\end{equation*}
(the equality requires some computation). After intersection with
$V\times T$ this gives
\begin{equation*}
   I=(xq,yq,tq, z)
\end{equation*}
and $(I:t)/(t) = (q,z)$. This is $C$.

\emph{Embedded point at a singularity.} Let $C\subset V$ be the union of
two distinct lines intersecting in $P$ and consider a planar embedded
point at $P$. Despite the singularity, there is still a $\PP^1$ of
embedded points at $P$. We take this to be our one parameter family,
i.e.\ we deform the embedded point structure away from the planar one.

In local coordinates we take $P$ to be the origin in $\Aff^3$ and $C$ to
be the union of the $x$- and $y$-axes in the $xy$-plane $V$. Then
\begin{equation*}
   (xy,z)(x,y,z) + (z-txy) = (xy^2,x^2y,z-txy)
\end{equation*}
is our one parameter family of embedded points at the origin, with $t=0$
corresponding to the planar embedded point. The intersection with
$V\times T$ is given by
\begin{equation*}
   I=(xy^2,x^2y,z,t xy)
\end{equation*}
and $(I:t)/(t) = (xy,z)$. This is $C$.

\emph{Embedded point in a double line.} Let $C\subset V$ be a planar
double line together with a planar embedded point at $P\in C$ and take
the one parameter family of embedded points in $P$.

In local coordinates we take $P$ to be the origin in $\Aff^3$ and $C$ to
be $V(z,y^2)$. Then
\begin{equation*}
   (z,y^2)(x,y,z) + (z-ty^2) = (y^3,xy^2,z-ty^2)
\end{equation*}
is our one parameter family of embedded points at the origin, with $t=0$
corresponding to the planar embedded point. The intersection with
$V\times T$ is given by
\begin{equation*}
   I=(xy^2,y^3,z,t y^2)
\end{equation*}
and $(I:t)/(t)=(z,y^2)$. This is $C$.
\end{proof}

\begin{proof}[Proof of Theorem \ref{thm:modulichamberII} (ii)]
The morphism $\hilb^{2m+2}(\PP^3) \to \MII$ is clearly an isomorphism
away from $E$, and it sends $\xi\in E$ (lying over $(P,V)\in I$) to
$\F_{P,V}$, which determines and is uniquely determined by $(P,V)$.
Moreover $\MII$ is smooth at these points by Proposition
\ref{pro:dimExtF}. The claim follows from uniqueness of normal
contractions.
\end{proof}

\subsection{Moduli in chamber III}\label{sec:chamberIII}

In this section we show that the moduli space $\MIII$ is a contraction
of $\mathcal S$. The argument parallels that for $\MII$ closely.

Let $F\subset \mathcal S$ be as in Notation \ref{notn:divisors}. Thus an
element $Y\in \mathcal S$ is either a pair of intersecting lines with a
spatial embedded point at the intersection, or as degenerate cases, a
planar double line with a spatial embedded point. It is in a natural way
a $\PP^2$-bundle over the incidence variety
$I\subset \PP^3\times\check\PP^3$ via the map
\begin{equation*}
   F \to I
\end{equation*}
that sends $Y$ to the pair $(P,V)$ consisting of the support $P\in Y$ of
the embedded point and the plane $V$ containing $Y\setminus\{P\}$. In
parallel with Proposition \ref{prop:negative-one} one may show that
$\OO_{\mathcal S}(F)$ restricts to $\OO_{\PP^2}(-1)$ in the fibers of
$F/I$ and so there is a contraction
\begin{equation}\label{eq:S-contraction}
   \psi\colon \mathcal S\to\mathcal S'
\end{equation}
to a smooth algebraic space $\mathcal S'$, such that $\psi$ is an
isomorphism away from $F$ and restricts to the $\PP^2$-bundle $F\to I$.
However, in this case we can be much more concrete thanks to the work of
Chen--Coskun--Nollet \cite{CCN2011}, where birational models for
$\mathcal S$ are studied in detail (and in greater generality: moduli
spaces for pairs of codimension two linear subspaces of projective
spaces in arbitrary dimension). The following proposition is
\cite[Theorem 1.6 (4)]{CCN2011}; we sketch a simple and slightly
different argument here.

\begin{pro}\label{prop:grassmann}
There is a contraction as in \eqref{eq:S-contraction} where
$\mathcal S'$ is the Grassmannian $G(2,6)$ of lines in $\PP^5$.
\end{pro}

\begin{proof}
First consider an arbitrary quadric $Q\subset \PP^n$. Any finite
subscheme in $Q$ of length $2$, reduced or not, determines a line in
$\PP^n$. This defines a morphism
\begin{equation}\label{eq:hilb2Q}
   \hilb^2(Q) \to G(2, n+1).
\end{equation}
It is clearly an isomorphism away from the locus in $G(2,n+1)$
consisting of lines contained in $Q$. On the other hand, over every
element of $G(2,n+1)$ defining a line contained in $Q$, the fiber is the
$\PP^2$ consisting of length two subschemes of that line.

Apply the above observation to the (Plücker) quadric $Q=G(2,4)$ in
$\PP^5$, so that $\mathcal S \iso \hilb^2(Q)$ (see Section
\ref{sec:twocomponents}). For every plane $V\subset \PP^3$ and every
point $P\in V$, the pencil of lines in $V$ through $P$ defines a line in
$Q=G(2,4)$ and in fact every line is of this form. The fiber of
\eqref{eq:hilb2Q} above such an element of $G(2,5)$ consists of all
pairs of lines in $V$ intersecting at $P$. It follows that
\eqref{eq:hilb2Q} is the required contraction
$\mathcal S \to \mathcal S'$.
\end{proof}

\begin{rem}
Chen--Coskun--Nollet furthermore shows that \eqref{eq:S-contraction} is
a $K$-negative extremal contraction in the sense of Mori theory. In
fact, $\mathcal S$ is Fano and its Mori cone is spanned by two rays.
Either ray is thus contractible; one contraction is
\eqref{eq:S-contraction} and the other is the natural map to the
symmetric square of the Grassmannian of lines in $\PP^3$. This statement
is extracted from Theorem 1.3, Lemma 3.2 and Proposition 3.3 in
\emph{loc.\ cit.} Inspired by this work we return to the Mori cone of
the conics-with-a-point component $\mathcal C$ in Section
\ref{sec:mori}.
\end{rem}

We proceed as for chamber II by modifying the universal family of pairs
of lines in order to identify the moduli space $\MIII$ with the
contracted space $\mathcal S'$. Let
\begin{equation*}
   \mathcal Y \subset \PP^3\times \mathcal S
\end{equation*}
be the restriction of the universal family over $\hilb^{2m+2}(\PP^3)$ to
the component $\mathcal S$. Moreover, there is a flat family over the
incidence variety $I\subset\PP^3\times\check\PP^3$ whose fiber over
$(P,V)$ is the plane $V$ with an embedded point at $P$. Pull this back
to $F$ to define a family
\begin{equation*}
   \mathcal W \subset \PP^3\times F.
\end{equation*}
We argue as in Section \ref{sec:familychamberII} but with the family of
planes $\mathcal V$ replaced by the family of planes with an embedded
point $\mathcal W$.

\begin{thm}\label{thm:modulichamberIII}
Let $\mathcal Y$ and $\mathcal W$ be as above.
\begin{itemize}
   \item[(i)]
   $\I_{\mathcal Y \cup \mathcal W}$ is flat as a coherent sheaf over
   $\mathcal S$. Its fibers $\I_{\mathcal Y\cup\mathcal W} \tensor
   k(\xi)$ over $\xi\in\mathcal S$ are stable objects for stability
   conditions in chamber III.
   \item[(ii)]
   The morphism $\mathcal S \to \MIII$ determined by $\I_{\mathcal Y
   \cup \mathcal W}$ induces an isomorphism $\mathcal S' \iso \MIII$.
\end{itemize}
\end{thm}

For $\xi\in F$ lying over $(P,V)$ we have
$\I_{\mathcal Y_\xi/ \mathcal W_\xi} \iso \OO_V(-2)$ and
$\I_{\mathcal W_\xi} \iso \I_P(-1)$. Thus Lemma \ref{lem:opposite-ses}
yields a short exact sequence
\begin{equation*}
   0 \to \OO_V(-2) \to \I_{\mathcal Y \cup \mathcal W}\tensor k(\xi)
   \to \I_P(-1) \to 0
\end{equation*}
and we show that it is nonsplit by exhibiting a certain quotient sheaf
of $\I_{\mathcal Y \cup \mathcal W}\tensor k(\xi)$. This time we use
$\I_{Q/V}(-1)$ where $Q\in V$ is a point distinct from $P$.

\begin{lem}\label{lem:nosurjectionIq/v}
Let $V\subset \PP^3$ be a plane and $P, Q\in V$ two distinct points.
There is no surjection from  $\OO_V(-2)\oplus \I_P(-1)$ to
$\I_{Q/V}(-1)$.
\end{lem}

\begin{proof}
Every nonzero homomorphism
\begin{equation*}
   \OO_V(-2) \to \OO_V(-1)
\end{equation*}
has image of the form $\I_{L/V}(-1)$ where $L\subset V$ is a line,
whereas every nonzero homomorphism
\begin{equation*}
   \I_P(-1) \to \OO_V(-1)
\end{equation*}
has image $\I_{P/V}(-1)$. Thus any nonzero homomorphism from the direct
sum of these two sheaves has image $\I(-1)$ where $\I\subset \OO_V$ is
one of $\I_{L/V}$, $\I_{P/V}$ or their sum
\begin{equation*}
   \I_{L/V} + \I_{P/V} =
   \begin{cases}
   \I_{P/V} & \text{if $P\in L$}\\
   \OO_V & \text{otherwise.}
   \end{cases}
\end{equation*}
Thus the image is never $\I_{Q/V}(-1)$ for $Q\ne P$.
\end{proof}

\begin{proof}[Proof of Theorem \ref{thm:modulichamberIII}]
The proof for Theorem \ref{thm:modulichamberII} carries over; we only
need to detail the construction of quotient sheaves via one parameter
families. As before we write down families over $\Aff^1=\Spec k[t]$ and
then restrict to $T = \Spec k[t]/(t^2)$. We then apply Lemma
\ref{lem:quotientexists1}, with $\mathcal W$ in the role of the family
denoted $\mathcal V$ in the Lemma. The outcome of Lemma
\ref{lem:quotientexists1} will be a quotient sheaf of the form
$\I_{Y'/W}$, where $W = \mathcal W_\xi$ is a plane $V$ with an embedded
point at $P$. We end by intersecting with $V$ to produce a further
quotient of the form $\I_{Y'\cap V/V}$. We shall choose one parameter
families such that the latter is isomorphic to $\I_{Q/V}(-1)$ with
$Q\ne P$.

\emph{Distinct lines.} Let $C = L \cup L_0$ be a pair of distinct lines
inside $V$ intersecting at $P$. Choose another plane $V'$ containing
$L_0$ and a point $Q\in L_0$ distinct from $P$. The pencil of lines
$L_t \subset V'$ through $Q$ yields a one parameter family
\begin{equation*}
   Z_t = L \cup L_t
\end{equation*}
of disjoint pairs of lines for $t\ne 0$, with flat limit $Z_0\subset W$
being $C$ with a spatial embedded point at $P$.

In suitable affine coordinates $\Aff^3$ let $V$ be $V(z)$, let $P$ be
the origin and let $C=V(z,xy)$. Then $W = V(xz,yz,z^2)$. Furthermore let
$Q=(0,1,0)$ and $L_t = V(x,z-t(y-1))$. This leads to the family $Z$
defined by the ideal
\begin{equation*}
   (y,z) \cap (x,z-t(y-1))=(xy,xz,(z-t(y-1))y, (z-t(y-1))z)
\end{equation*}
and the intersection with $W\times T$ is given by
\begin{equation*}
   I = (xz,yz,z^2,xy,t y(y-1),t z)
\end{equation*}
Thus $(I:t)/(t) = (z,xy,y(y-1))$, which defines the union of the
$x$-axis and the point $Q$. This is $Y'\subset W$ and thus
$\I_{Y'\cap V/V} = \I_{Y'/V}$ is isomorphic to $\I_{Q/V}(-1)$.

\emph{Double lines.} Let $C\subset V$ be a double line inside the plane
$V$ with $P\in C$. We shall define an explicit one parameter family with
central fiber $Y_0\subset W$ being $C$ with a spatial embedded point at
$P$.

Geometrically, the family is this: let $L\subset V$ be the supporting
line of $C$. Consider a line $M\subset V$ not through $P$ and let $Q$ be
its intersection point with $L$. Also let $M'$ be a line through $P$ and
not contained in $V$. Let $R_t$ be a point on $M$ moving towards $Q$ as
$t\to 0$, and let $R'_t$ be a point on $M'$ moving towards $P$, but much
faster than $R_t$ moves (quadratic versus linearly). Then let $L_t$ be
the line through $R_t$ and $R'_t$ and let $Y_t = L\cup L_t$ for
$t\ne 0$.

Let $P$ be the origin in suitable affine coordinates $\Aff^3$, let $V$
be the $xy$-plane $V(z)$ and let $C\subset V$ be the double $x$-axis
$V(y^2,z)$. Thus $Y_0$ corresponds to
\begin{equation*}
   (z,y^2)\cap (x,y,z)^2 = (xz,yz,z^2,y^2).
\end{equation*}
Now let $L=V(y,z)$, let $L_t$ be the line through $(1,t,0)$ and
$(0,0,t^2)$, that is
\begin{equation*}
   L_t = V(tx-y, ty+z-t^2)
\end{equation*}
and take $Y_t = L\cup L_t$ for $t\ne 0$. This yields the family (the
following identity requires a bit of fiddling)
\begin{equation*}
   (y,z) \cap (tx-y,ty+z-t^2)
   = ((tx-y)y, (tx-y)z, (ty+z-t^2)z, xz + ty(x-1)).
\end{equation*}
Reducing this modulo $t$ gives the original $Y_0$. The intersection with
$W\times T$ gives
\begin{equation*}
   I = (xz,yz,z^2,(t x-y)y, t y(x-1))
\end{equation*}
and so $Y'\subset W$ is defined by
\begin{align*}
   (I:t)/(t) &= (xz,yz,z^2,y^2,y(x-1))\\
   &= (y,z) \cap (x-1,y^2,z) \cap (x,y,z^2).
\end{align*}
This is the line $L$ with an embedded point at $Q$ (inside $V$) and
another embedded point along the $z$-axis at $P$. Intersecting with $V$
removes the embedded point at $P$, leaving the line $L$ with an embedded
point at $Q$. Thus $\I_{Y'\cap V/V} \iso \I_{Q/V}(-1)$.

This establishes part (i) precisely as in the proof of Theorem
\ref{thm:modulichamberII} and part (ii) then follows by smoothness of
$\mathcal{M}^{\mathrm{III}}$ (from Proposition \ref{pro:dimExtG}) and by
uniqueness of normal (here smooth) contractions.
\end{proof}


\section{The Mori cone of $\mathcal C$ and extremal contractions}\label{sec:mori}

In this final section we shall prove that $\mathcal C \to \mathcal C'$
is the contraction of a $K$-negative extremal ray in the Mori cone. It
follows that the contracted space $\mathcal C'$ is projective.

To set the stage we recall the basic mechanism of $K$-negative extremal
contractions. Let $X$ be a projective normal variety and $\alpha$ a
curve class (modulo numerical equivalence) which spans an extremal ray
in the Mori cone. If also the ray is $K$-negative, i.e.\ the
intersection number between $\alpha$ and the canonical divisor $K_X$ is
negative, then there exists a unique projective normal variety $Y$ and a
birational morphism $f\colon X\to Y$ which contracts precisely the
effective curves in the class $\alpha$.

\subsection{Statement}

We denote elements in $\mathcal C$ by the letter $Y$. It is the union of
a (possibly degenerate) conic denoted $C$ and a point denoted $P$. If
the point is embedded, $P\in C$ denotes its support. We also write
$V\subset \PP^3$ for the unique plane containing $C$.

Define four effective curve classes (modulo numerical equivalence) on
$\mathcal C$. Each is described as a family $\{Y_t\}$, and we use a
subscript $t$ to indicate a parameter on the piece that varies (all
choices are to be made general, e.g.\ $C$ nonsingular unless stated
otherwise, etc.):
\begin{enumerate}
   \item[$\delta$:] fix a conic $C$ and a point $P\in C$. Let $Y_t$ be
   $C$ with an embedded point at $P$, varying in the $\PP^1$ of normal
   directions to $C \subset \PP^3$ at $P$.
   \item[$\epsilon$:] fix a plane $V$, a conic $C\subset V$ and a line
   $L\subset V$. Let $P_t$ vary along $L$ and let
   $Y_t = C \cup \{P_t\}$.
   \item[$\zeta$:] fix a plane $V$, a pencil of conics $C_t \subset V$
   and a point $P\in V$. Let $Y_t = C_t \cup \{P\}$.
   \item[$\eta$:] fix a line $L$ and a point $P\in L$. Let $V_t$ be the
   pencil of planes containing $L$ and let $C_t$ be the planar double
   structure on $L$ inside $V_t$. Then let $Y_t$ be $C_t$ with an
   embedded spatial point at $P$.
\end{enumerate}
In $\epsilon$ there are implicitly two elements with an embedded point,
namely where $L$ intersects $C$. Similarly there is one element in
$\zeta$ with an embedded point, corresponding to the pencil member $C_t$
that contains $P$.

\begin{thm}\label{thm:mori-C}
The Mori cone of $\mathcal C$ is the cone over a solid tetrahedron, with
extremal rays spanned by the four curve classes $\delta, \epsilon,
\zeta, \eta$. Of these, the first three are $K$-negative, whereas $\eta$
is $K$-positive. The contraction corresponding to $\zeta$ is
$\mathcal C \to \mathcal C'$.
\end{thm}

\begin{cor}
$\mathcal C'$ is projective.
\end{cor}

\begin{rem}
The last claim in the theorem is clear: by contracting $\zeta$ we forget
the conic part of $Y\subset V$, keeping only $V$ and the point $P\in V$.
By uniqueness of (normal) contractions the contracted variety is
$\mathcal C'$. Also, with reference to Diagram \ref{eq:Ccontr} (from
Section \ref{sec:modulispace:contraction}), the contraction of $\delta$
is the blowing down $b$. The theorem furthermore reveals a third
$K$-negative extremal ray spanned by $\epsilon$. The corresponding
contraction has the effect of forgetting the point part of $Y\subset V$,
keeping only the conic; thus the contracted locus in $\mathcal C$ is the
same as for $\zeta$, but the contraction happens in a ``different
direction''. We do not know if the contracted space has an
interpretation as a moduli space for Bridgeland stable objects.
\end{rem}

\subsection{The canonical divisor}

Use notation as in Diagram \ref{eq:Ccontr} and Lemma \ref{lem:P5bundle}.
We read off that the Picard group of $\mathcal C$ has rank $4$ and is
generated by the pullbacks of the following divisor classes:
\begin{align*}
   &H \subset \PP^3 & &\text{a plane,} \\
   &H' \subset \check \PP^3 & &\text{a plane in the dual space,} \\
   &A = c_1(\OO_{\PP(\E^\vee)}(1)), \\
   &E' \subset \mathcal C & &\text{the exceptional divisor for the blowup
   $b$.}
\end{align*}
Moreover numerical and linear equivalence of divisors coincide on
$\mathcal C$. Here we only use that the Picard group of a projective
bundle over some variety $X$ is $\Pic(X)\oplus \ZZ$, with the added
summand generated by $\OO(1)$, and the Picard group of a blowup of $X$
is $\Pic(X)\oplus \ZZ$, with the added summand generated by the
exceptional divisor.

As long as confusion seems unlikely to occur we will continue to use the
symbols $H$, $H'$ and $A$ for their pullbacks to $\mathcal C$, or to an
intermediate variety such as $\PP^3\times\hilb^{2m+2}(\PP^3)$ in Diagram
\ref{eq:Ccontr}.

\begin{lem}\label{lem:KC}
The canonical divisor class of $\mathcal C$ is
\begin{equation*}
   K_{\mathcal C} = -4H - 8H' - 6A + E'.
\end{equation*}
\end{lem}

\begin{proof}
This is a standard computation. First, for the blowup $b$, with center
of codimension two, we have
\begin{equation*}
   K_{\mathcal C} = b^*K_{\PP^3 \times \hilb^{2m+1}(\PP^3)} + E'
\end{equation*}
and for the product
\begin{equation*}
   K_{\PP^3 \times \hilb^{2m+1}(\PP^3)} = \pr_1^*K_{\PP^3} +
   \pr_2^*K_{\hilb^{2m+1}(\PP^3)}.
\end{equation*}
Now $K_{\PP^3} = -4H$ and for the projective bundle
$\hilb^{2m+1}(\PP^3) \iso \PP(\E^\vee)$ we have
\begin{equation*}
   K_{\PP(\E^\vee)} = \pi^*K_{\check\PP^3} + c_1(\Omega^1_\pi).
\end{equation*}
Again $K_{\check\PP^3} = -4H'$ and the short exact sequence
\begin{equation*}
   0 \to \Omega^1_\pi \to \pi^*(\E^\vee) \tensor \OO_{\PP(\E^\vee)}(-1)
   \to \OO_{\PP(\E^\vee)} \to 0
\end{equation*}
gives
\begin{align*}
   c_1(\Omega^1_\pi)
   &= c_1(\pi^*(\E^\vee) \tensor \OO_{\PP(\E^\vee)}(-1)) \\
   &= \pi^*c_1(\E^\vee) + 6c_1(\OO_{\E^\vee}(-1) \\
   &= -\pi^*c_1(\E) - 6A.
\end{align*}
Putting this together, the stated expression for $K_{\mathcal C}$
follows once we have established that $c_1(\E) = 4H'$.

Recall that $\E = \pr_{2*}(\res{\pr_1^*\OO_{\PP^3}(2)}{I})$. We compute
its first Chern class by brute force: apply Grothendieck--Riemann--Roch
to $\pr_2\colon \PP^3\times\check\PP^3 \to \check\PP^3$. Note that all
higher direct images vanish, since $H^p(V,\OO_V(2)) = 0$ for all
$V\in \check\PP^3$ and $p>0$. Thus by Grothendieck--Riemann--Roch the
class
\begin{equation*}
   c_1(\pr_{2*}(\res{\pr_1^*\OO_{\PP^3}(2)}{I}))
\end{equation*}
is the push forward in the sense of the Chow ring of the degree $4$
homogeneous part of
\begin{equation*}
   \ch(\res{\pr_1^*\OO_{\PP^3}(2)}{I}) \pr_1^*(\td(\PP^3)).
\end{equation*}
We have
\begin{equation}\label{eq:chE1}
   \td(\PP^3) =
   \left(\frac{H}{1-e^{-H}}\right)^4 = 1 + 2H + \tfrac{11}{6} H^2 + H^3.
\end{equation}
Moreover $I\subset \PP^3\times\check\PP^3$ is a divisor of bidegree
$(1,1)$, so there is a short exact sequence
\begin{equation*}
   0 \to \pr_1^*\OO_{\PP^3}(-1) \otimes \pr_2^*\OO_{\check\PP^3}(-1)
   \to \OO_{\PP^3\times\check\PP^3} \to \OO_I \to 0
\end{equation*}
from which we see (suppressing the explicit pullbacks $\pr_i^*$ of
cycles in the notation)
\begin{equation}\label{eq:chE2}
   \ch(\res{\pr_1^*\OO_{\PP^3}(2)}{I}))
   = \exp(2H)(1-\exp(-H)\exp(-H')).
\end{equation}
Now multiply together \eqref{eq:chE1} and \eqref{eq:chE2} and observe
that the $H^3H'$-coefficient is $4$. Since the push forward $\pr_{2*}$
of any degree $4$ monomial $H^kH'^{4-k}$ equals $H'$ if $k=3$ and $0$
otherwise, this shows that $c_1(\E) = 4H'$.
\end{proof}

\subsection{Basis for $1$-cycles}

We will need a few more effective curves, as before written as families
$\{Y_t\}$:
\begin{enumerate}
   \item[$\alpha:$] fix a conic $C$ and a line $L$. Let the point $P_t$
   vary along $L$ and let $Y_t = C \cup \{P_t\}$.
   \item[$\beta:$] fix a quadric surface $Q\subset \PP^3$, a line $L$
   and a point $P$. Let $V_t$ run through the pencil of planes
   containing $L$ and let $C_t = Q\cap V_t$. Then take
   $Y_t = C_t \cup \{P\}$.
   \item[$\gamma:$] fix a plane $V$ and a point $P$. Let $C_t\subset V$
   run through a pencil of conics and let $Y_t = C_t \cup \{P\}$.
\end{enumerate}
As before all choices are general, so that in the definition of
$\alpha$, the line $L$ is disjoint from $C$, etc.

\begin{lem}\label{lem:dualbasis}
The dual basis to $(H, H', A, E')$ is $(\alpha,\beta,\gamma,-\delta)$.
\end{lem}

\begin{proof}
We need to compute all the intersection numbers and verify that we get
$0$ or $1$ as appropriate. Here it is sometimes useful to explicitly
write out the pullbacks to $\mathcal C$, e.g.\ writing $b^*(\pr_1^*(H))$
rather than $H$. We view $\alpha,\beta,\gamma,\delta$ not just as
equivalence classes, but as the effective curves defined above. Only the
intersection numbers involving $\beta$ require some real work, and we
will save this for last.

Intersections with $\alpha$: Since $\pr_{1*}(b_*(\alpha)))$ is the line
$L\subset\PP^3$ defining $\alpha$ we have $b^*(\pr_1^*(H))\cdot \alpha =
H\cdot L = 1$. Similarly $\pr_{2*}(b_*(\alpha))=0$ shows that the
intersections with $H'$ and $A$ vanish. Finally $\alpha$ has no elements
with embedded points, so is disjoint from $E'$.

Intersections with $\gamma$: We have $A\cdot\gamma=1$ because $\gamma$
is a line in a fiber of the projective bundle $\pi$, whereas $A$
restricts to a hyperplane in every fiber. The remaining intersection
numbers vanish as we can pick disjoint effective representatives.

Intersections with $\delta$: We have $E'\cdot\delta = -1$ as $\delta$ is
a fiber of the blowup $b$ and $E'$ is the exceptional divisor. The
remaining divisors $H$, $H'$ and $A$ are all pullbacks, i.e.\ of the
form $b^*(?)$ and then $b^*(?)\cdot \delta = (?)\cdot b_*(\delta) = 0$.

Intersections with $\beta$: We can choose $\beta$ to be disjoint from
$H$ and $E'$. Moreover $\pi_*(\pr_{1*}(b_*(\beta)))$ is the line $\check
L \subset\check\PP^3$ dual to the line $L$ defining $\beta$. This gives
$b^*(\pr_1^*(\pi^*(H')))\cdot \beta = H' \cdot \check L = 1$.

It remains to verify $A\cdot\beta=0$. The definition of $\beta$ can be
understood as follows: choose a general section
\begin{equation*}
   \OO_{\PP^3} \xrightarrow{\sigma} \OO_{\PP^3}(2)
\end{equation*}
and apply 
$\pr_{2*}(\res{\pr_1^*(-)}{I})$ to obtain a homomorphism
\begin{equation}\label{eq:sublinebundle}
   \OO_{\check \PP^3} \to \E
\end{equation}
whose fiber over $V\in\check\PP^3$ is exactly the restriction of
$\sigma$ to $V$. This is nowhere zero, so \eqref{eq:sublinebundle} is a
rank $1$ subbundle and it defines a section
\begin{equation*}
   s_Q\colon \check\PP^3\to \PP(\E^{\vee}) \iso \hilb^{2m+1}(\PP^3)
\end{equation*}
with $s_Q^*(\OO_{\PP(\E^\vee)}(-1)) \iso \OO_{\check\PP^3}$ or in terms
of divisors $s_Q^*(A) = 0$. If we let $Q\subset\PP^3$ be the quadric
defined by $\sigma$ then $s_Q(V) = Q \cap V$. Thus
$\pr_{2*}(b_*(\beta))) = s_{Q*}(\check L)$ where $\check L$ is the dual
to the line $L$ defining $\beta$. This gives
\begin{equation*}
   b^*(\pr_2^*(A)) \cdot \beta
   = A \cdot \pr_{2*}(b_*(\beta))
   = A \cdot s_{Q*}(\check L)
   = s_Q^*(A)\cdot \check L = 0.
\end{equation*}
\end{proof}

We also define the following three effective divisors, phrased as a
condition on $Y\in \mathcal C$:
\begin{enumerate}
   \item[$D:$] all $Y$ whose conic part $C$ intersects a
   fixed line $M\subset\PP^3$.
   \item[$D':$] all $Y$ such that the line through $P$ and
   a fixed point $P_0\in \PP^3$ intersects the conic part $C$.
   \item[$E:$] all planar $Y$ (as before).
\end{enumerate}
Since $D$ is defined by a condition on $C$ only, it is the preimage by
$\pr_2\circ b$ (see Diagram \ref{eq:Ccontr}) of the similarly defined
divisor in $\hilb^{2m+1}(\PP^3)$. Moreover $D'$ and $E$ are the
\emph{strict} transforms by $b$ of the similarly defined divisors on
$\PP^3\times\hilb^{2m+1}(\PP^3)$.

We will need to control elements of $D'$ with an embedded point.

\begin{lem}\label{lem:D'-embedded}
Fix $P_0$ so that $D'$ is defined as an effective divisor. Choose a
plane $V$ not containing $P_0$, a possibly degenerate conic $C\subset V$
and a point $P\in C$. Then there is a unique $Y\in D'$ with conic part
$C$ and an embedded point at $P$. More precisely:
\begin{enumerate}
   \item If $C$ is nonsingular at $P$ then the embedded point structure
   is uniquely determined by the normal direction given by the line
   through $P_0$ and $P$.
   \item If $C$ is a pair of lines intersecting at $P$ or a double line,
   then the embedded point is the spatial one, i.e.\ the scheme
   theoretic union of $C$ and the first order infinitesimal neighborhood
   of $P$ in $\PP^3$.
\end{enumerate}
\end{lem}

\begin{proof}
Let $Q$ be the cone over $C$ with vertex $P_0$. This is a quadratic cone
in the usual sense when $C$ is nonsingular, otherwise $Q$ is either a
pair of planes or a double plane. A disjoint union $C \cup \{P'\}$ with
$P'\ne P_0$ is clearly in $D'$ if and only if it is a subscheme of $Q$.

On the one hand this shows that the subschemes $Y$ listed in (1) and (2)
are indeed in $D'$, since they are obtained from $C\cup\{P'\}$ by
letting $P'$ approach $P$ along the line joining $P_0$ and $P$.

On the other hand it follows that if $Y\in D'$ then $Y\subset Q$, since
the latter is a closed condition on $Y$. In case (1) $Q$ is nonsingular
at $P$ and so there is a unique embedded point structure at $P\in C$
which is contained in $Q$. In case (2) the following explicit
computation gives the result: suppose in local affine coordinates that
$C$ is the pair of lines $V(xy,z)$, the ``vertex'' $P_0$ is on the
$z$-axis and $P$ is the origin. Then $Q$ is the pair of planes $V(xy)$.
Any $C$ with an embedded point at $P$ has ideal of the form
\begin{equation*}
   (xy,z)(x,y,z) + (sxy + tz)
\end{equation*}
for $(s:t)\in \PP^1$. This contains the defining equation $xy$ of $Q$ if
and only if $t=0$, which defines the spatial embedded point. The case
where $C$ is double line $V(x^2,z)$ is similar.
\end{proof}

\begin{lem}\label{lem:div-relations}
We have
\begin{align*}
   D &= 2H' + A,\\
   D' &= 2H + 2H' + A - E',\\
   E &= H + H' - E'.
\end{align*}
\end{lem}

\begin{proof}
The last equality was essentially established in the proof of
Proposition \ref{prop:negative-one}: it follows from the observations
(1) $E$ is the strict transform of $b(E)$, and (2) the latter is the
pullback of the incidence variety $I\subset \PP^3\times\check\PP^3$
which is linearly equivalent to $H + H'$.

The remaining two identities are verified by computing the intersection
numbers with the curves in the basis from Lemma \ref{lem:dualbasis}. All
curves and divisors involved are concretely defined and it is easy to
find and count the intersections directly. Some care is needed to rule
out intersection multiplicities, and we often find it most efficient to
resort to a computation in local coordinates. We limit ourselves to
writing out only two cases.

The case $D\cdot \beta = 2$: As we noted $D$ is really a divisor on
$\hilb^{2m+1}(\PP^3)$ and so we shall write it here as
$b^*(\pr_2^*(D))$. Then $D\subset\hilb^{2m+1}(\PP^3)$ consists of all
conics intersecting a fixed line $M$. We have
\begin{equation*}
   b^*(\pr_2^*(D)) \cdot \beta = D \cdot \pr_{2*}(b_*(\beta))
\end{equation*}
and $\pr_{2*}(b_*(\beta))$ is the family of conics $C_t = V_t \cap Q$
where $Q$ is a fixed quadric surface and $V_t$ runs through the pencil
of planes containing a fixed line $L$. For general choices $M\cap Q$
consists of two points, and each point spans together with $L$ a plane.
This yields exactly two planes $V_0$ and $V_1$ in the pencil for which
$C_0 = V_0\cap Q$ and $C_1 = V_1\cap Q$ intersects $M$. It remains to
rule out multiplicities.

In the local coordinates in Remark \ref{rem:localcontr} let $M =
V(x_0,x_1)$. Then the intersection between $M$ and the plane $x_3 =
c_0x_0 + c_1x_1 + c_2x_2$ is the point $(0:0:1:c_2)$. Now $D$ is the
condition that this point is on $C$, i.e.\ it satisfies equation
\eqref{eq:Clocal2}; this gives that $D$ is $s_{22}=0$. On the other
hand, $\pr_{2*}(b_*(\beta))$ is a one parameter family in which $c_i$
and $s_{ij}$ are functions of degree at most $2$ in the parameter. To
stay concrete, let $Q$ be $\sum_i x_i^2 = 0$ and let $V_t$ be $x_3 =
tx_2$. Substitute $x_3 = tx_2$ in the equation for $Q$ to find $C_t =
Q\cap V_t$. This gives in particular $s_{22} = 1 + t^2$ and so the
intersection with $D$ is indeed two distinct points, each of
multiplicity $1$.

The case $D'\cdot \delta = 1$: This is essentially Lemma
\ref{lem:D'-embedded}, but to ascertain there is no intersection
multiplicity to account for we argue differently. $D'$ is the strict
transform of the divisor $b(D') \subset \PP^3\times\hilb^{2m+1}(\PP^3)$,
which contains the center of the blowup. Since $b(D')$ is nonsingular
(pick $P_0 = (0:0:0:1)$ in the definition of $D'$, then in the local
coordinates of Remark \ref{rem:localcontr} it is simply given by the
equation \eqref{eq:Clocal2}) we have $D' = b^*(b(D')) - E'$. Thus
\begin{equation*}
   D'\cdot\delta = b(D')\cdot b_*(\delta) - E'\cdot\delta = 0 - (-1)
\end{equation*}
and we are done.

The remaining cases are either similar to these or easier.
\end{proof}

\subsection{Nef and Mori cones}

It is clear that $H$, $H'$ and $D$ are base point free, hence nef. For
instance, consider $D$: given $Y\in \mathcal C$, choose a line
$M\subset \PP^3$ disjoint from $Y \subset \PP^3$. This defines an
effective representative for $D$ not containing $Y$.

\begin{lem}\label{lem:nef}
The divisor $D'+H'$ is nef.
\end{lem}

\begin{proof}
We begin by narrowing down the base locus of $D'$. First consider an
element $Y\in\mathcal C$ without embedded point, that is a disjoint
union $Y = C \cup \{P\}$. Then choose $P_0$ such that the line through
$P_0$ and $P$ is disjoint from $C$. This defines a representative for
$D'$ not containing $Y$, so $Y$ is not in the base locus.

Next let $Y$ be a conic $C$ with an embedded point at a point $P\in C$
where $C$ is nonsingular. The tangent to $C$ at $P$ together with the
normal direction given by the embedded point determines a plane. Pick
$P_0$ such that the line through $P$ and $P_0$ defines a normal
direction to $C$ which is distinct from that defined by the embedded
point. This determines a representative for $D'$ which by Lemma
\ref{lem:D'-embedded}(i) does not contain $Y$, so $Y$ is not in the base
locus.

The remaining possibility is that $Y$ is either a pair of intersecting
lines with an embedded point at the singularity, or a double line with
an embedded point. Pick a representative for $D'$ by choosing $P_0$
outside the plane containing the degenerate conic. If the embedded point
is not spatial, then Lemma \ref{lem:D'-embedded}(ii) shows that $Y$ is
not in $D'$. So $Y$ is not in the base locus unless the embedded point
is spatial.

Thus let $B \subset \mathcal C$ be the locus of intersecting lines with
a spatial embedded point at the origin, together with double lines with
a spatial embedded point. By the above $B$ contains the base locus of
$D'$, so if $T\subset \mathcal C$ is an irreducible curve not contained
in $B$ then
\begin{equation*}
   (D' + H')\cdot T = D'\cdot T + H' \cdot T \ge 0
\end{equation*}
as both terms are nonnegative. If on the other hand $T\subset B$ we
observe that $T\cdot E = 0$: in fact $B$ and $E$ are disjoint, since
every element in $B$ has a spatial embedded point, whereas all elements
in $E$ are planar. By the relations in Lemma \ref{lem:div-relations}
\begin{equation*}
   D' + H' = H + D + E
\end{equation*}
and so, using that $H$ and $D$ are nef,
\begin{equation*}
   (D' + H')\cdot T = (H + D + E)\cdot T
   = \underbrace{H\cdot T + D\cdot T}_{\ge 0}
   + \underbrace{E\cdot T}_{0} \ge 0.
\end{equation*}
\end{proof}

\begin{lem}\label{lem:nefmori}
The dual basis to $(H, H', D, D'+H')$ is $(\epsilon, \eta, \zeta,
\delta)$.
\end{lem}

\begin{proof}
Lemma \ref{lem:dualbasis} together with the relations in Lemma
\ref{lem:div-relations} implies that $(D'+H')\cdot \delta = 1$ and the
other tree intersection numbers with $\delta$ vanish.

Of the remaining intersection numbers only those involving $\eta$
requires some care and we shall write out only those.

A representative for $\eta$ is obtained by fixing a line $L$ and a point
$P\in L$ and letting the plane $V_t$ vary in the pencil of planes
containing $V_t$. Then $C_t$ is the double $L$ inside $V_t$ and $Y_t$ is
$C_t$ together with an embedded spatial point at $P$. Then:
\begin{itemize}
   \item Intersecting with $H$ imposes the condition that $P$ is
   contained in a fixed but arbitrary plane, but $P$ is fixed, so
   $H\cdot \eta = 0$.
   \item Intersecting with $H'$ imposes the condition that $V_t$
   contains a fixed but arbitrary point $P_0$, this gives
   $H'\cdot \eta = 1$. (In fact this can be identified with the
   intersection number $H'\cdot \check L = 1$ in $\check \PP^3$, where
   $\check L$ is the dual line to $L$, so there is no subtlety regarding
   transversality of the intersection.)
   \item Intersecting with $D$ imposes the condition that $C_t$
   intersects a fixed but arbitrary line $M$, but $C_t$ has fixed
   support $L$, so $D\cdot \eta = 0$.
\end{itemize}
As $\eta$ is contained in the base locus of $D'$ we cannot find
$D'\cdot\eta$ directly. As in the proof of Lemma \ref{lem:nef} we
instead rewrite $D' + H'$ as $H + D + E$ and take advantage of $\eta$
being disjoint from $E$. This gives
\begin{equation*}
   (D'+H') \cdot \eta = (H+D+E)\cdot \eta = 0.
\end{equation*}
\end{proof}

We wish to point out that the computation in the very last paragraph,
showing $D'\cdot \eta = -1$, is what made us realize that the addition
of $H'$ is necessary to produce a nef divisor.

\begin{proof}[Proof of Theorem \ref{thm:mori-C}]
The four divisors in Lemma \ref{lem:nefmori} are nef (the first three
are base point free, and the fourth is treated in Lemma \ref{lem:nef})
and the four curves are effective by definition. Hence they span the nef
and Mori cones of $\mathcal C$, respectively. Finally by Lemmas
\ref{lem:KC} and \ref{lem:div-relations} we have
\begin{equation*}
   K = -2H + 5H' - 5D - (D'+H')
\end{equation*}
so in view of the dual bases in Lemma \ref{lem:nefmori} we read off that
$K$ is negative on $\epsilon$, $\zeta$, $\delta$ and positive on $\eta$.
\end{proof}


\bibliographystyle{plain}
\bibliography{bibliography}

\begin{thebibliography}{10}

\bibitem{AS2023}
S.~Alaoui~Soulimani.
\newblock {\em Bridgeland Stability Conditions and the {H}ilbert Scheme of Skew
  Lines in Projective Space}.
\newblock Phd thesis, University of Stavanger, March 2023.
\newblock Available at \url{https://hdl.handle.net/11250/3058550}.

\bibitem{Art1970}
M.~Artin.
\newblock Algebraization of formal moduli: {II}. {E}xistence of modifications.
\newblock {\em Annals of Mathematics}, 91(1):88--135, 1970.

\bibitem{BM2011}
A.~Bayer and E.~Macrì.
\newblock The space of stability conditions on the local projective plane.
\newblock {\em Duke Math. J.}, 160(2):263--322, 2011.

\bibitem{BMS2016}
A.~Bayer, E.~Macrì, and P.~Stellari.
\newblock The space of stability conditions on {A}belian threefolds, and on
  some {C}alabi-{Y}au threefolds.
\newblock {\em Inventiones mathematicae}, 206(3):869--933, 2016.

\bibitem{BMT2014}
A.~Bayer, E.~Macrì, and Y.~Toda.
\newblock Bridgeland stability conditions on threefolds {I}:
  {B}ogomolov-{G}ieseker type inequalities.
\newblock {\em J. Algebraic Geom.}, 23(1):117--163, 2014.

\bibitem{Bri2007}
T.~Bridgeland.
\newblock Stability conditions on triangulated categories.
\newblock {\em Annals of Mathematics}, 166(2):317--345, 2007.

\bibitem{Bri2008}
T.~Bridgeland.
\newblock Stability conditions on {K3} surfaces.
\newblock {\em Duke Mathematical Journal}, 141(2):241 -- 291, 2008.

\bibitem{CCN2011}
D.~Chen, I.~Coskun, and S.~Nollet.
\newblock Hilbert scheme of a pair of codimension two linear subspaces.
\newblock {\em Comm. Algebra}, 39(8):3021--3043, 2011.

\bibitem{CN2012}
D.~Chen and S.~Nollet.
\newblock Detaching embedded points.
\newblock {\em Algebra Number Theory}, 6(4):731--756, 2012.

\bibitem{FN1971}
A.~Fujiki and S.~Nakano.
\newblock Supplement to ``{O}n the {I}nverse of {M}onoidal {T}ransformation.
\newblock {\em Publ. Res. Inst. Math. Sci. (Kyoto)}, 7(3):637--644, 1971.

\bibitem{GLS2018}
P.~Gallardo, C.~Lozano~Huerta, and B.~Schmidt.
\newblock Families of elliptic curves in $\mathbb{P}^{3}$ and {B}ridgeland
  stability.
\newblock {\em Michigan Mathematical Journal}, 67(4):787 -- 813, 2018.

\bibitem{Las1969}
A.~T. Lascu.
\newblock Sous-variétés régulièrement contractibles d'une variété
  algébrique.
\newblock {\em Annali della Scuola Normale Superiore di Pisa - Classe di
  Scienze}, 23(4):675--695, 1969.

\bibitem{Lee2000}
Y.~H.~A. Lee.
\newblock {\em The {H}ilbert scheme of curves in {$\mathbb{P}^3$}}.
\newblock PhD thesis, Harvard university, 2000.

\bibitem{Mac2014}
E.~Macrì.
\newblock A generalized {B}ogomolov{-}{G}ieseker inequality for the
  three{-}dimensional projective space.
\newblock {\em Algebra Number Theory}, 8(1):173--190, 2014.

\bibitem{MS2017}
E.~Macrì and B.~Schmidt.
\newblock Lectures on {B}ridgeland stability.
\newblock In {\em Moduli of curves}, volume~21 of {\em Lect. Notes Unione Mat.
  Ital.}, pages 139--211. Springer, Cham, 2017.

\bibitem{Moi1967}
B.~G. Moishezon.
\newblock On $n$-dimensional compact complex manifolds having $n$ algebraically
  independent meromorphic functions.
\newblock {\em Am. Math. Soc.}, 63:51--177, 1967.

\bibitem{Nak1970}
S.~Nakano.
\newblock On the {I}nverse of {M}onoidal {T}ransformation.
\newblock {\em Publ. Res. Inst. Math. Sci. (Kyoto)}, 6(3):483--502, 1970.

\bibitem{PT2019}
D.~Piyaratne and Y.~Toda.
\newblock Moduli of {B}ridgeland semistable objects on 3-folds and
  {D}onaldson-{T}homas invariants.
\newblock {\em J. Reine Angew. Math. (Crelle's Journal)}, 747:175--219, 2019.

\bibitem{Rez2021}
F.~Rezaee.
\newblock Geometry of canonical genus four curves, 2021.
\newblock Preprint arXiv:2107.14213 [math.AG].

\bibitem{Sch2020}
B.~Schmidt.
\newblock Bridgeland stability on threefolds: some wall crossings.
\newblock {\em J. Algebraic Geom.}, 29(2):247--283, 2020.

\bibitem{stacks-project}
The {Stacks project authors}.
\newblock The stacks project.
\newblock \url{https://stacks.math.columbia.edu}, 2023.

\bibitem{Xia2018}
B.~Xia.
\newblock Hilbert scheme of twisted cubics as a simple wall-crossing.
\newblock {\em Trans. Amer. Math. Soc.}, 370(8):5535--5559, 2018.

\end{thebibliography}

\end{document}